\documentclass[reqno]{amsart}
\usepackage{enumerate}
\usepackage{amssymb,url,color, booktabs}
\usepackage{mathrsfs}
\usepackage{amsmath}
\usepackage{color}

\usepackage{algorithm}
\usepackage{algpseudocode}

\usepackage{booktabs}
\usepackage{multirow}
\usepackage{subcaption}

\usepackage[colorlinks=true,linkcolor=blue,citecolor=blue]{hyperref}

\usepackage{bbm, framed}
\usepackage{tikz}
\usepackage{threeparttable}
\usepackage{mathtools}
\usepackage{comment}
\usepackage[colorlinks=true]{hyperref}
\hypersetup{
linkcolor=blue,          % color of internal links
citecolor=red,        % color of links to bibliography
filecolor=blue,      % color of file links
urlcolor=cyan
}
\usepackage{xcolor}

\numberwithin{equation}{section}

\newcommand{\be}{\begin{eqnarray}}
\newcommand{\ee}{\end{eqnarray}}
\newcommand{\ce}{\begin{eqnarray*}}
\newcommand{\de}{\end{eqnarray*}}
\newtheorem{theorem}{Theorem}[section]
\newtheorem{lemma}[theorem]{Lemma}
\newtheorem{remark}[theorem]{Remark}
\newtheorem{definition}[theorem]{Definition}
\newtheorem{proposition}[theorem]{Proposition}
\newtheorem{Examples}[theorem]{Example}
\newtheorem{corollary}[theorem]{Corollary}

\def\var{{\mathrm{var}}}
\def\HS{{\mathrm{HS}}}
\def\eps{\varepsilon}

\def\<{{\langle}}
\def\>{{\rangle}}
\def\({{\Big(}}
\def\){{\Big)}}

\def\bx{{\mathbf{x}}}

\def\dif{{\mathord{{\rm d}}}}

\def\no{\nonumber}
\def\={&\!\!=\!\!&}

\def\cB{{\mathcal B}}

\def\cF{{\mathcal F}}

\def\cH{{\mathcal H}}

\def\cN{{\mathcal N}}

\def\mE{{\mathbb E}}

\def\mN{{\mathbb N}}

\def\mP{{\mathbb P}}

\def\mR{{\mathbb R}}

\def\1{{\mathbf{1}}}

\def\sI{{\mathscr I}}
\def\sJ{{\mathscr J}}

\def\E{\mathbb E}

\def\ge{\geqslant}
\def\le{\leqslant}

\def\var{{\mathrm{var}}}
\def\HS{{\mathrm{HS}}}
\def\eps{\varepsilon}

\def\<{{\langle}}
\def\>{{\rangle}}
\def\({{\Big(}}
\def\){{\Big)}}

\def\bx{{\mathbf{x}}}

\def\dif{{\mathord{{\rm d}}}}

\def\no{\nonumber}
\def\={&\!\!=\!\!&}
\def\bt{\begin{theorem}}
\def\et{\end{theorem}}
\def\bl{\begin{lemma}}
\def\el{\end{lemma}}
\def\br{\begin{remark}}
\def\er{\end{remark}}
\def\bx{\begin{Examples}}
\def\ex{\end{Examples}}
\def\bd{\begin{definition}}
\def\ed{\end{definition}}
\def\bp{\begin{proposition}}
\def\ep{\end{proposition}}
\def\bc{\begin{corollary}}
\def\ec{\end{corollary}}

\def\ge{\geqslant}
\def\le{\leqslant}

 \def\R{\mathbb R}
 \def\R{\mathbb R}    
\def\N{\mathbb N}  
   
\def\<{\langle} \def\>{\rangle}

\allowdisplaybreaks

\begin{document}

\title[Strong approximation by compound Poisson processes]
{Strong approximation for stochastic Volterra equations by compound Poisson processes}

\date{}

\author{Xicheng Zhang,\ \ Yuanlong Zhao}

%
%\address{Xicheng Zhang:
%School of Mathematics and Statistics, Beijing Institute of Technology, Beijing 100081, China\\
%Faculty of Computational Mathematics and Cybernetics, Shenzhen MSU-BIT University, 518172 Shenzhen, China\\
%Email: xczhang.math@bit.edu.cn
%}
%
%\address{Yuanlong Zhao:
%School of Mathematics and Statistics, Wuhan University,
%Wuhan, Hubei 430072, P.R.China\\
%Email:yuanlong.zhao@whu.edu.cn
%}

\thanks{
Xicheng Zhang: School of Mathematics and Statistics, Beijing Institute of Technology, Beijing 100081, China; 
Faculty of Computational Mathematics and Cybernetics, Shenzhen MSU-BIT University, 518172 Shenzhen, China. Email: xczhang.math@bit.edu.cn\\[0.3em]
\mbox{}\hspace{1.2em}Yuanlong Zhao: School of Mathematics and Statistics, Wuhan University, Wuhan, Hubei 430072, P.R.China. Email: yuanlong.zhao@whu.edu.cn\\[0.3em]
\mbox{}\hspace{1.2em}This work is supported by National Key R\&D program of China (No. 2023YFA1010103) and NNSFC grant of China (No. 12131019, 12595282)  and the DFG through the CRC 1283 ``Taming uncertainty and profiting from randomness and low regularity in analysis, stochastics and their applications''. Y. Zhao is also supported by the China Scholarship
Council (Grant 202406270165).}

\begin{abstract}
We study a \emph{compound Poisson (random time-change)} approximation for stochastic differential equations (SDEs) and stochastic Volterra equations whose coefficients may be merely measurable in time and may even exhibit integrable singularities.
For an SDE driven by Brownian motion, we replace the time variable by the Poisson clock $\cN_t^\eps$ and approximate the stochastic integral by $W_{\cN_t^\eps}$, which leads to an explicit jump scheme driven by a compensated Poisson random measure.
Under standard Lipschitz/linear-growth conditions in the state variable (with no continuity assumed in time for the drift), we prove \emph{strong convergence} and obtain explicit rates in $\eps$.
For Volterra-type equations with singular kernels, we establish strong convergence as well, with a rate that reflects both the temporal regularity of the kernel and the intrinsic $\eps^{1/2}$ fluctuation of the Poisson clock.
The compound Poisson scheme differs fundamentally from the Euler--Maruyama method: it does not require pointwise evaluation of time-irregular coefficients on a deterministic grid, and it remains stable in the presence of time singularities.
We further illustrate the theory on stochastic Volterra equations driven by fractional Brownian motion and provide numerical experiments showing improved performance over Euler--Maruyama for problems with singular time dependence.
\end{abstract}

\maketitle \rm

%\tableofcontents

\section{Introduction}

Stochastic differential equations (SDEs) and stochastic Volterra equations (SVEs) arise ubiquitously in models with
multiscale structure \cite{PavliotisStuart08,FreidlinWentzell12},
rough temporal features, and memory effects \cite{gripenberg1990volterra,BergerMizel80}.
In many applications the coefficients are not smooth in time: the drift may be merely measurable,
may have jump discontinuities (for instance, due to regime switching) \cite{Zvonkin74,Veretennikov81,KrylovRockner05,FlandoliGubinelliPriola10,YinZhu10},
or may even display \emph{integrable} singularities of the form $|t-t_0|^{-\alpha}$ \cite{Zhang08Volterra,Tudor95}.
Such temporal irregularities pose a concrete challenge for classical time stepping methods, most notably the
Euler--Maruyama (EM) scheme, whose stability and strong error analysis are typically carried out under
additional time-regularity assumptions (often of H\"older type) \cite{KP92,higham2002strong}.

\medskip

This paper develops a different discretization principle based on \emph{time randomization}.
Instead of sampling the coefficients on a deterministic mesh, we drive the numerical scheme by a Poisson clock
and evaluate Brownian increments at random times.
This leads to a simple and implementable \emph{compound Poisson approximation} that is particularly robust when
the coefficients are irregular in time.
Heuristically, a random clock is unlikely to repeatedly sample a small set of ``bad'' times (discontinuities or spikes);
more importantly, the key error terms can be controlled via moment bounds for the clock fluctuation,
rather than by pointwise time continuity of the drift.

\medskip

We begin with the SDE on $\mR^d$,
\begin{align}\label{sde_0}
\dif X_t=\sigma(t,X_t)\,\dif W_t+b(t,X_t)\,\dif t,\qquad t\ge 0,
\end{align}
where $(W_t)_{t\ge 0}$ is an $m$-dimensional Brownian motion on a complete filtered probability space
$(\Omega,\cF,\{\cF_t\}_{t\ge 0},\mP)$.
The coefficients $\sigma:\mR_+\times\mR^d\to\mR^d\otimes\mR^m$ and $b:\mR_+\times\mR^d\to\mR^d$ are Borel measurable and satisfy
a time-inhomogeneous global Lipschitz and linear growth condition.

\begin{description}
\item[{\bf (H$^p_0$)}] For some $p\ge1$, there exists $\ell_b\in\bigcap_{T>0}L^{2p}([0,T])$ 
such that for all $t\ge 0$,
\begin{align}\label{as_lwm_rewrite}
\sup_{x\neq y}\frac{|b(t,x)-b(t,y)|}{|x-y|}\le \ell_b(t),
\qquad
\sup_x\frac{|b(t,x)|}{1+|x|}\le \ell_b(t),
\end{align}
and there exists $\kappa_\sigma>0$ such that for all $t\ge 0$,
\begin{align}\label{as_sigma}
\sup_{x\neq y}\frac{\|\sigma(t,x)-\sigma(t,y)\|_{\mathrm{HS}}}{|x-y|}
+
\sup_x\frac{\|\sigma(t,x)\|_{\mathrm{HS}}}{1+|x|}
\le \kappa_\sigma,
\end{align}
where $\|\sigma\|_{\mathrm{HS}}^2:=\sum_{i=1}^d\sum_{j=1}^m|\sigma_{ij}|^2$.
\end{description}
Under {\bf (H$^p_0$)}, since $\ell_b\in L^2_{\mathrm{loc}}(\R_+)$ on every finite time interval, \eqref{sde_0} admits a unique strong solution; see, e.g., \cite{liu2015stochastic}.

\medskip

To approximate \eqref{sde_0} in the strong sense, the EM scheme samples the coefficients on a deterministic time grid. Given a deterministic time discretization $0 = t_0 < t_1 < t_2 < \cdots$, the EM method constructs an approximate solution $\{Y_{t_k}\}$ recursively: starting from $Y_{t_0} = X_0$, for $k \in \mathbb{N}$,
\begin{align*}
Y_{t_{k+1}} = Y_{t_k} + b(t_k, Y_{t_k})(t_{k+1} - t_k) + \sigma(t_k, Y_{t_k})(W_{t_{k+1}} - W_{t_k}).
\end{align*}
This discretization extends the classical Euler method to the stochastic setting by approximating the Itô integral via Brownian increments.

\medskip

While this choice is natural and widely used, it can become delicate when the time dependence is highly irregular.
Indeed, if $t\mapsto b(t,\cdot)$ is merely measurable, has jump discontinuities, or exhibits integrable singularities,
then a fixed mesh may place many evaluation points in neighborhoods where the coefficients change abruptly.
From an analytical perspective, most strong error estimates for EM rely on controlling temporal increments of the coefficients and therefore
typically impose some form of time regularity (often of H\"older type), in addition to spatial Lipschitz and growth conditions; see, e.g.,
\cite{KP92,higham2002strong}.
Recent developments have relaxed the \emph{spatial} regularity requirements in various directions, while still retaining
nontrivial temporal regularity assumptions to close the strong error estimates; see, for instance, \cite{yan2022irregular} and references therein.
Moreover, when coefficients have superlinear growth, the classical EM scheme may fail to converge in $L^p$ \cite{hutzenthaler2011strong},
which has motivated a broad literature on stabilized explicit methods such as tamed Euler schemes \cite{hutzenthaler2012strong,10.1214/ECP.v18-2824}.
Overall, despite its simplicity, the accuracy and stability of EM-type discretizations are closely tied to temporal regularity
and growth properties of the coefficients, motivating alternative discretization principles for models with irregular time dependence.

\subsection*{Compound Poisson time randomization}

In contrast, we randomize the time mesh using a Poisson clock.
Let $(T_k)_{k\in\mN}$ be i.i.d.\ exponential random variables with parameter $1$ and define the Poisson process
\[
\cN_t:=\max\Bigl\{n:\sum_{k=1}^n T_k\le t\Bigr\},\qquad t\ge 0.
\]
For $\eps>0$, set
\[
\cN_t^\eps:=\eps\,\cN_{t/\eps},\qquad \widetilde{\cN}_t^\eps:=\cN_t^\eps-t.
\]
Then $\cN^\eps$ has jump size $\eps$ and jump intensity $1/\eps$, while $\widetilde{\cN}^\eps$ is a martingale and
\[
\mE(\cN_t^\eps-t)^2=\var(\cN_t^\eps)=\eps t.
\]
The time-changed process $(W_{\cN_t^\eps})_{t\ge 0}$ is an $\mR^m$-valued compound Poisson process whose jumps are Gaussian:
each jump has law $N(0,\eps I_m)$, i.e.\ its L\'evy measure is
\begin{align}\label{Nor1}
\nu_\eps(\dif z)=(2\pi\eps)^{-m/2}\exp\Bigl(-\frac{|z|^2}{2\eps}\Bigr)\,\dif z,\qquad z\in\mR^m.
\end{align}

We approximate \eqref{sde_0} by the compound Poisson scheme
\begin{align}\label{sde_eps}
X_t^\eps
=
X_0
+\int_0^t \sigma(s,X_{s-}^\eps)\,\dif W_{\cN_s^\eps}
+\int_0^t b(s,X_{s-}^\eps)\,\dif \cN_s^\eps .
\end{align}
Equivalently, writing $S_k^\eps:=\eps\sum_{i=1}^k T_i$ for the jump times of $\cN^\eps$, we obtain the fully discrete representation
\[
X_t^\eps
=
X_0
+\sum_{k=1}^{\cN_{t/\eps}}
\Bigl[
\sigma(S_k^\eps,X_{S_{k-1}^\eps}^\eps)\bigl(W_{k\eps}-W_{(k-1)\eps}\bigr)
+\eps\, b(S_k^\eps,X_{S_{k-1}^\eps}^\eps)
\Bigr],
\]
so the method can be viewed as a random time discretization that is straightforward to implement.

\medskip

Random time discretizations have appeared in several forms (e.g.\ weak approximation and random grid methods); see, for instance,
\cite{zhang2025compound}.
Our focus is on \emph{strong} approximation under \emph{time-irregular} coefficients, including the case where the drift has no continuity
assumption in the time variable.

\medskip

\noindent\textbf{Main contributions.}
The contributions of this work are as follows.

\begin{itemize}
\item \textbf{Strong convergence with explicit rates.}
We establish strong convergence of the compound Poisson approximation for both standard SDEs and stochastic Volterra equations.
The rates are explicit and depend on quantitative indices describing the temporal regularity/singularity of the coefficients (and, for Volterra equations, of the kernel).

\item \textbf{Irregular-in-time drift.}
In contrast to the classical EM framework, our analysis does not require the drift to be continuous in time and accommodates jump discontinuities and integrable time singularities.
This is particularly relevant for models with temporal spikes, regime switching, or rough forcing.

\item \textbf{Volterra equations: discretization adapted to the two-time structure.}
For stochastic Volterra equations, a direct transplant of the SDE argument fails for the clock error term due to the two-time dependence $(t,s)$ in the kernel.
We introduce a discretization tailored to the Volterra structure and derive a corresponding strong error bound.

\item \textbf{Examples and numerics.}
We verify the assumptions for a Volterra equation associated with fractional Brownian motion and provide numerical experiments.
The results illustrate that the proposed Poisson scheme remains stable and accurate in settings with singular time dependence, where EM-type schemes can become less reliable.
\end{itemize}

\subsection*{Main result for SDEs}

To state our first result, we impose a (possibly non-uniform) time continuity assumption on the diffusion coefficient $\sigma$.

\begin{description}
\item[{\bf (H$^t_\sigma$)}] There exist $\alpha,\beta\in(0,1]$ and $\kappa_\tau>0$ such that for all $t,s\in\mR_+$ and $x\in\mR^d$,
\[
\|\sigma(t,x)-\sigma(s,x)\|^2_{\mathrm{HS}}
\le
\kappa_\tau\,|t^\alpha-s^\alpha|^\beta\,(1+|x|^2).
\]
\end{description}
\begin{theorem}\label{1}
Assume that {\bf (H$^p_0$)} and {\bf (H$^t_\sigma$)} hold.
Then for any $T>0$ and $X_0\in L^{2p}(\Omega)$, there exists $C>0$ such that for all $\eps\in(0,1)$ and $t\in[0,T]$,
\[
\mE|X_t^\eps-X_t|^{2p}
\le
C\,\eps^{\frac{\beta p}{2}}.
\]
\end{theorem}

\begin{remark}\label{rem:optimality}
Suppose that for all $\theta,\theta'>0$ and $x,y\in\mR^d$,
\[
|h(\theta,x)-h(\theta',y)|
\le C\bigl(|\theta-\theta'|^{1/2}+|x-y|\bigr).
\]
Then for any $\alpha\in(0,1]$, the diffusion coefficient $\sigma(t,x):=h(t^\alpha,x)$ satisfies {\bf (H$^t_\sigma$)} with $\beta=1$, and hence
\[
\mE|X_t^\eps-X_t|^{2p}\le C\,\eps^{\frac p2}.
\]
Moreover, this rate is optimal in general. Indeed, letting $\Delta_t^\eps:=\cN_t^\eps-t$, conditional on $\cN_t^\eps$ we have
$W_{\cN_t^\eps}-W_t\sim N(0,|\Delta_t^\eps|\,I_m)$, and thus
\[
\mE|W_{\cN_t^\eps}-W_t|^4
=
m(m+2)\,\mE|\Delta_t^\eps|^2
=
m(m+2)\,\eps t,
\]
since $\mE|\cN_t^\eps-t|^2=\var(\cN_t^\eps)=\eps t$.
\end{remark}

\subsection*{Stochastic Volterra equations}

We next consider the stochastic Volterra equation (SVE)
\begin{align}\label{sve}
Y_t=Y_0+\int_0^t \sigma(t,s,Y_s)\,\dif W_s+\int_0^t b(t,s,Y_s)\,\dif s,
\end{align}
where the kernel-type dependence in $(t,s)$ may be singular.
To formulate our conditions, we introduce the following notation. For any $\delta > 0$ and $u \ge 0$, we define the left-endpoint projection
\begin{align}\label{De1}
    u_\delta := k\delta = \lfloor u/\delta \rfloor \delta \quad \text{for } u \in [k\delta, (k+1)\delta) \text{ with integer } k \ge 0.
\end{align}
In particular, for the time variable $t$, we simply write 
$$
t_\delta := \lfloor t/\delta \rfloor \delta.
$$

We work under assumptions {\bf (H$^\gamma_1$)}--{\bf (H$^\gamma_3$)} stated below, which allow for time-singular behavior through integrability
conditions on suitable envelope functions $\ell_i$.

\begin{description}
\item[{\bf (H$^\gamma_1$)}] There exist $\ell_1,\ell_2:\mR_+\times\mR_+\to\mR_+$ such that for all $0\le s<t<\infty$,
\begin{align}
\sup_{x\neq y}\frac{\|\sigma(t,s,x)-\sigma(t,s,y)\|^2_{\mathrm{HS}}}{|x-y|^2}
+\sup_x\frac{\|\sigma(t,s,x)\|^2_{\mathrm{HS}}}{1+|x|^2}
&\le \ell_1(t,s), \label{h11}\\
\sup_{x\neq y}\frac{|b(t,s,x)-b(t,s,y)|^2}{|x-y|^2}
+\sup_x\frac{|b(t,s,x)|^2}{1+|x|^2}
&\le \ell_2(t,s), \label{h12}
\end{align}
and for any $T>0$ there exist $\gamma\in(0,1]$ and $C>0$ such that for all $t\in(0,T]$ and $s\in(0,t\wedge 1)$,
\begin{align}\label{h13}
\int_0^s\ell_1(t,r)\,\dif r + \int_{t-s}^t \ell_1(t,r)\,\dif r \le C s^\gamma,
\end{align}
and for all $\delta\in(0,1)$,
\begin{align}\label{h130}
\int_0^t(\ell_1(t,r)+\ell_1(t,r_\delta)+\ell_2(t,r))\,\dif r\le C.
\end{align}
\end{description}

\begin{description}
\item[{\bf (H$^\gamma_2$)}] There exist $\ell_3,\ell_4:\mR_+\times\mR_+\times\mR_+\to\mR_+$ such that for all $t,t'\in\mR_+$ and $0\le s<t\wedge t'$,
\begin{align}
\sup_x\frac{\|\sigma(t',s,x)-\sigma(t,s,x)\|^2_{\mathrm{HS}}}{1+|x|^2}
&\le \ell_3(t',t,s), \label{h21}\\
\sup_x\frac{|b(t',s,x)-b(t,s,x)|^2}{1+|x|^2}
&\le \ell_4(t',t,s), \label{h22}
\end{align}
and for any $T>0$ there exist $\gamma\in(0,1]$ and $C>0$ such that for all $0\le t<t'\le T$,
\begin{align}\label{h23}
\int_0^t\bigl(\ell_3(t',t,s)+\ell_4(t',t,s)\bigr)\,\dif s
\le C|t'-t|^\gamma.
\end{align}
\end{description}

\begin{description}
\item[{\bf (H$^\gamma_3$)}] There exists $\ell_5:\mR_+\times\mR_+\times\mR_+\to\mR_+$ such that for all $0\le s,s'<t$,
\begin{align}\label{h31}
\sup_x\frac{\|\sigma(t,s,x)-\sigma(t,s',x)\|^2_{\mathrm{HS}}}{1+|x|^2}
\le \ell_5(t,s,s'),
\end{align}
and for any $T>0$ there exist $\gamma\in(0,1]$ and $C>0$ such that for all $t\le T$ and $0<\delta<1\wedge t/3$,
\begin{align}\label{h32}
\int_\delta^{t_\delta}\ell_5(t,s,s_\delta)\,\dif s \le C\delta^\gamma.
\end{align}
\end{description}

Under {\bf (H$^\gamma_1$)} and {\bf (H$^\gamma_2$)}, \eqref{sve} admits a unique solution; see, e.g., \cite{wang08volterra,BergerMizel80,zhang2010stochastic}.
While EM-type schemes for SVEs have been widely studied, most existing works focus on non-singular kernels; see \cite{Tudor95,wen2011improved,wang2017approximate}.
A general framework for weakly singular kernels was developed in \cite{Zhang08Volterra}, and
$\theta$-EM and Milstein-type schemes were investigated in \cite{li2022numerical}.

For SVEs, the compound Poisson approximation takes the form
\[
Y_t^\eps
=
Y_0
+\int_0^t \sigma(t,s,Y_{s-}^\eps)\,\dif W_{\cN_s^\eps}
+\int_0^t b(t,s,Y_{s-}^\eps)\,\dif \cN_s^\eps,
\]
or equivalently,
\[
Y_t^\eps
=
Y_0
+\sum_{k=1}^{\cN_{t/\eps}}
\Bigl[
\sigma(t,S_k^\eps,Y_{S_{k-1}^\eps}^\eps)\bigl(W_{k\eps}-W_{(k-1)\eps}\bigr)
+\eps\, b(t,S_k^\eps,Y_{S_{k-1}^\eps}^\eps)
\Bigr].
\]

\begin{theorem}\label{2}
Under {\bf (H$^\gamma_1$)}, {\bf (H$^\gamma_2$)}, and {\bf (H$^\gamma_3$)} with $\gamma\in(0,1]$,
for any $T>0$ and $Y_0\in L^{2}(\Omega,\cF_0)$ there exists a constant $C>0$ such that for all $\eps\in(0,1)$ and $t\in[0,T]$,
\[
\mE|Y_t^\eps-Y_t|^2 \le C\,\eps^{\gamma/(2(2+\gamma))}.
\]
\end{theorem}

\medskip

The remainder of the paper is organized as follows.
Sections~2 and~3 contain the proofs of Theorems~\ref{1} and~\ref{2}, respectively.
Section~4 verifies {\bf (H$^\gamma_1$)}--{\bf (H$^\gamma_3$)} for a Volterra equation associated with fractional Brownian motion and derives an explicit rate.
Section~5 reports numerical experiments, including comparisons with the classical EM scheme, highlighting the improved stability of the proposed method in the presence of time singularities.

\section{Proof of Theorem~\ref{1}}\label{Holder}

We begin by preparing two technical lemmas.
\begin{lemma}\label{lem_x}
For any $\alpha\in(0,1]$ and $\beta\ge 0$, there exists a constant $C=C(\alpha,\beta)>0$ such that for any $\eps>0$ and $k\in\N$,
\begin{align}\label{es_x}
\E\big|r^\alpha-(S_k^\eps)^\alpha\big|^\beta \le C\,(k\eps)^{\alpha\beta}\,k^{-\beta/2},
\qquad r\in[(k-1)\eps,k\eps].
\end{align}
\end{lemma}

\begin{proof}
Recall $S_k^\eps=\eps S_k$ with $S_k=\sum_{i=1}^k T_i$ and $T_i\sim\mathrm{Exp}(1)$ i.i.d.
By scaling,
\[
\E\big|r^\alpha-(S_k^\eps)^\alpha\big|^\beta
=\eps^{\alpha\beta}\,\E\big|(r/\eps)^\alpha-S_k^\alpha\big|^\beta,
\]
so it suffices to prove \eqref{es_x} for $\eps=1$.

Fix $k\ge1$ and $r\in[k-1,k]$. We use the elementary inequality: for all $a,b>0$ and $\alpha\in(0,1]$,
\begin{equation}\label{ineq_power}
|a^\alpha-b^\alpha|\le a^{\alpha-1}|a-b|.
\end{equation}
(For $\alpha=1$ this is an equality.) Applying \eqref{ineq_power} with $a=r$ and $b=S_k$ yields
\[
|r^\alpha-S_k^\alpha|^\beta \le r^{(\alpha-1)\beta}|r-S_k|^\beta.
\]
Since $r\in[k-1,k]$, we have $r\sim k$ and $|r-S_k|\le |r-k|+|k-S_k|\le 1+|S_k-k|$, hence
\begin{equation}\label{AA1_new}
\E|r^\alpha-S_k^\alpha|^\beta
\lesssim k^{(\alpha-1)\beta}\bigl(1+\E|S_k-k|^\beta\bigr).
\end{equation}

Finally, $S_k-k=\sum_{i=1}^k(T_i-1)$ is a sum of i.i.d.\ centered random variables with finite variance.
By the discrete Burkholder/Rosenthal inequality,
\[
\E|S_k-k|^\beta
\le \Big(\E\Big|\sum_{i=1}^k(T_i-1)\Big|^{\beta\vee 2}\Big)^{\beta/(\beta\vee 2)}
\lesssim \Big(\sum_{i=1}^k \E|T_i-1|^2\Big)^{\beta/2}
\lesssim k^{\beta/2}.
\]
Substituting this into \eqref{AA1_new} gives
\[
\E|r^\alpha-S_k^\alpha|^\beta \lesssim k^{(\alpha-1)\beta}k^{\beta/2}
= k^{\alpha\beta-\beta/2},
\]
which is \eqref{es_x} for $\eps=1$. Rescaling back completes the proof.
\end{proof}

\begin{remark}
If one only uses $|a^\alpha-b^\alpha|\le |a-b|^\alpha$ for $\alpha\in(0,1)$, then
\[
\E|r^\alpha-S_k^\alpha|^\beta \le \E|r-S_k|^{\alpha\beta}\lesssim k^{\alpha\beta/2},
\]
which is weaker than the rate $k^{\alpha\beta-\beta/2}$ from Lemma~\ref{lem_x}.
\end{remark}

\smallskip

Next, we estimate the deviation of $\cN_t^\eps$ from $t$.

\begin{lemma}\label{lem_y}
For any $p>0$, there exists a constant $C=C(p)>0$ such that for all $t>0$ and $\eps\in(0,1)$,
\begin{align}\label{es_y}
\E|\cN_t^\eps-t|^p \le C\,\eps^{p/2}\,(t^{p/2}\vee t).
\end{align}
\end{lemma}

\begin{proof}
Let $M_t:=\cN_t-t$. Then $(M_t)_{t\ge 0}$ is a càdlàg martingale with jumps
\[
\Delta M_s=\Delta \cN_s\in\{0,1\}.
\]
In particular, $|\Delta M_s|^p=\Delta \cN_s$ for every $s$, and the predictable quadratic variation satisfies
$\langle M\rangle_t=t$.

\medskip
\noindent\textbf{Case 1: $p\ge 2$.}
By the Burkholder--Davis--Gundy inequality for purely discontinuous martingales,
\begin{align*}
\E|M_t|^p &\lesssim\E\langle M\rangle_t^{p/2}+\E\sum_{0<s\le t}|\Delta M_s|^p
=t^{p/2}+\E\cN_t=t^{p/2}+t.
\end{align*}
Now use $\cN_t^\eps=\eps\,\cN_{t/\eps}$, so $\cN_t^\eps-t=\eps(\cN_{t/\eps}-t/\eps)=\eps M_{t/\eps}$ and hence
\[
\E|\cN_t^\eps-t|^p=\eps^p\,\E|M_{t/\eps}|^p
\lesssim\eps^p\big((t/\eps)^{p/2}+t/\eps\big).
\]
Since $p\ge2$ and $\eps\in(0,1)$, we have $\eps^{p-1}\le \eps^{p/2}$. Therefore
\[
\eps^p\big((t/\eps)^{p/2}+t/\eps\big)
\le \eps^{p/2}t^{p/2}+\eps^{p/2}t
\le \eps^{p/2}(t^{p/2}\vee t),
\]
which proves \eqref{es_y} for $p\ge2$.

\medskip
\noindent\textbf{Case 2: $p\in(0,2)$.}
By H\"older's inequality,
\[
\E|\cN_t^\eps-t|^p \le \big(\E|\cN_t^\eps-t|^2\big)^{p/2}.
\]
Since $\cN_{t/\eps}$ is Poisson with parameter $t/\eps$, we have
\[
\E|\cN_t^\eps-t|^2=\var(\cN_t^\eps)=\eps^2\var(\cN_{t/\eps})=\eps^2(t/\eps)=\eps t.
\]
Hence
\[
\E|\cN_t^\eps-t|^p \le (\eps t)^{p/2}=\eps^{p/2}t^{p/2}\le \eps^{p/2}(t^{p/2}\vee t).
\]
Combining the two cases completes the proof.
\end{proof}

\medskip

We first recall the following standard moment estimates for the strong solution $X$ of \eqref{sde_0}; see, e.g., \cite{Ik-Wa1981}.

\begin{lemma}\label{Le23}
Assume {\bf(H$^p_0$)} and fix $T>0$. Let $(X_t)_{t\ge0}$ be the unique strong solution to \eqref{sde_0}. Then there exists a constant
$
C=C\bigl(p,T,\kappa_\sigma,\|\ell_b\|_{L^2(0,T)}\bigr)>0
$
such that
\begin{align}\label{X_bound}
\E\left(\sup_{t\in[0,T]}|X_t|^{2p}\right)\le C\bigl(1+\E|X_0|^{2p}\bigr),
\end{align}
and  for all $s,t\in[0,T]$,
\begin{align}\label{D1}
\E|X_t-X_s|^{2p}\le C\bigl(1+\E|X_0|^{2p}\bigr)|t-s|^{p}.
\end{align}
\end{lemma}

We next establish a uniform moment bound for the approximation process $X^\eps$.
\begin{lemma}\label{bound}
Under {\bf(H$^p_0$)}, for any $X_0\in L^{2p}(\Omega)$ and $T>0$, there exists a constant
$C=C(p,,\|\ell_b\|_{L^{2p}(0,T)},\kappa_\sigma,m,T)>0$ such that for all $t\in[0,T]$
\begin{align}\label{X^eps_bound}
\sup_{\eps\in(0,1)}\E|X_t^\eps|^{2p}\le C\bigl(1+\E|X_0|^{2p}\bigr).
\end{align}
\end{lemma}

\begin{proof}
Let $\cH^\eps$ be the Poisson random measure associated with the jumps of
$W_{\cN^\eps_\cdot}$, i.e. for $t>0$ and $E\in\cB(\R^m)$,
\begin{align}\label{def_H}
\cH^\eps([0,t],E)
:=\sum_{0<s\le t}\1_E(\Delta W_{\cN_s^\eps})
=\sum_{k\le \cN_t^\eps/\eps}\1_E\big(W_{k\eps}-W_{(k-1)\eps}\big).
\end{align}
Its compensator is $(\dif s)\,\nu_\eps(\dif z)/\eps$, where $\nu_\eps$ is given by \eqref{Nor1}.
Denote the compensated measure by
\begin{align}\label{def_wt_H}
\widetilde{\cH}^\eps(\dif s,\dif z)
:=\cH^\eps(\dif s,\dif z)-\frac{\dif s}{\eps}\nu_\eps(\dif z).
\end{align}
Using $\int_{\R^m} z\,\nu_\eps(\dif z)=0$ and $\int_0^t \dif \cN_s^\eps=t+\widetilde{\cN}_t^\eps$
(with $\widetilde{\cN}_t^\eps:=\cN_t^\eps-t$), we can rewrite \eqref{sde_eps} as
\begin{align}\label{Xeps_comp}
X_t^\eps
&=X_0+\int_0^t b(s,X_s^\eps)\,\dif s
+\int_0^t\int_{\R^m}\Big(\sigma(s,X_{s-}^\eps)z+\eps b(s,X_{s-}^\eps)\Big)\,
\widetilde{\cH}^\eps(\dif s,\dif z).
\end{align}
Set
\[
\xi^\eps_s(z):=\sigma(s,X_s^\eps)z+\eps b(s,X_s^\eps).
\]
For $t>0$, by
Cauchy--Schwarz, and the BDG inequality for compensated Poisson integrals, we obtain
\begin{align*}
\E|X_t^\eps|^{2p}
&\lesssim \E|X_0|^{2p}
+ \E\Big(\int_0^t |b(s,X_s^\eps)|\,\dif s\Big)^{2p}\\
&\quad+ \E\Big(\int_0^t\int_{\R^m}|\xi^\eps_s(z)|^2\frac{\nu_\eps(\dif z)}{\eps}\dif s\Big)^p\\
&\quad+\E\Big(\int_0^t\int_{\R^m}|\xi^\eps_s(z)|^{2p}\frac{\nu_\eps(\dif z)}{\eps}\dif s\Big).
\end{align*}
By {\bf(H$^p_0$)}, $|b(s,x)|\le \ell_b(s)(1+|x|)$ and $\|\sigma(s,x)\|_{\rm HS}\le \kappa_\sigma(1+|x|)$. For the drift term, by H\"older's inequality, we have
\begin{equation}\label{bound_b(X)}
\E\Big(\int_0^t |b(s,X_s^\eps)|\,\dif s\Big)^{2p}
\lesssim\|\ell_b\|_{L^2(0,t)}^{2p}\int_0^t\Big(1+\mathbb E|X_r^\eps|^{2p}\Big)\dif r.
\end{equation}
Since
\[
|\xi^\eps_s(z)|\le\Big(\kappa_\sigma|z|+\ell_b(s)\eps\Big)(1+|X_s^\eps|),
\]
we have
\[
\int_{\R^m}|\xi^\eps_s(z)|^2\frac{\nu_\eps(\dif z)}{\eps}
\lesssim(1+|X_s^\eps|^2)\left(\eps+\int_{\R^m}|z|^2/\eps\,{\nu_\eps(\dif z)}\right)
\lesssim(1+|X_s^\eps|^2).
\]
So by Hölder's inequality,
\[
\E\Big(\int_0^t\int_{\R^m}|\xi^\eps_s(z)|^2\frac{\nu_\eps(\dif z)}{\eps}\dif s\Big)^p
\lesssim\int_0^t(1+\E|X_s^\eps|^{2p})\dif s.
\]
By similar argument, we also have
\[
\E\Big(\int_0^t\int_{\R^m}|\xi^\eps_s(z)|^{2p}\frac{\nu_\eps(\dif z)}{\eps}\dif s\Big)
\lesssim\int_0^t(1+\E|X_s^\eps|^{2p})\dif s.
\]
Combining the estimates above, we have
\[
\E|X_t^\eps|^{2p}\lesssim1+\E|X_0|^{2p}+\int_0^t\E|X_s^\eps|^{2p}\dif s.
\]
By Gronwall's inequality, we complete the proof.
\end{proof}

\medskip

To prove strong convergence, we use the decomposition
\begin{align*}
X_t-X_t^\eps
&=
\int_0^t\big(\sigma(s,X_s)-\sigma(s,X_{s\wedge \cN_s^\eps})\big)\,\dif W_s \\
&\quad
+\int_0^t \sigma(s,X_{s\wedge \cN_s^\eps})\,\dif\big(W_s-W_{\cN_s^\eps}\big)\\
&\quad
+\int_0^t\big(\sigma(s,X_{s\wedge \cN_s^\eps})-\sigma(s,X_{s-}^\eps)\big)\,\dif W_{\cN_s^\eps}\\
&\quad
+\Big(\int_0^t b(s,X_{s-}^\eps)\,\dif\cN_s^\eps-\int_0^t b(s,X_s)\,\dif s\Big)\\
&=:\sI_1(t)+\sI_2(t)+\sI_3(t)+\sI_4(t).
\end{align*}

We first estimate the simpler terms $\sI_1$ and $\sI_3$.
\begin{lemma}\label{I_1}
Under {\bf(H$^p_0$)}, for all $T>0$, there exists a constant $C>0$ such that for all $\eps\in(0,1)$ and $t\in[0,T]$,
\[
\E|\sI_1(t)|^{2p}
+\E|\sI_3(t)|^{2p}
\le C\Big(\eps^{\frac p2}+\int_0^t \E|X_s-X_s^\eps|^{2p}\,\dif s\Big).
\]
\end{lemma}

\begin{proof}
\noindent\textbf{Estimate of $\sI_1$.}
By BDG inequality, \eqref{as_sigma}, and Fubini's theorem,
\begin{align*}
\E|\sI_1(t)|^{2p}
&\lesssim\E\Big(\int_0^t
\|\sigma(s,X_s)-\sigma(s,X_{s\wedge \cN_s^\eps})\|_{\rm HS}^2\,\dif s\Big)^{p} \\
&\lesssim\E\Big(\int_0^t |X_s-X_{s\wedge \cN_s^\eps}|^{2}\,\dif s\Big)^{p}.
\end{align*}
Using the inequality $(\int_0^t f(s)\,\dif s)^p\le t^{p-1}\int_0^t f(s)^p\,\dif s$ for $f\ge0$,
\begin{align*}
\E|\sI_1(t)|^{2p}
&\lesssim \int_0^t \E|X_s-X_{s\wedge \cN_s^\eps}|^{2p}\,\dif s.
\end{align*}
Set $\tau_s:=s\wedge \cN_s^\eps$. Since $\cN^\eps$ is independent of $X$, conditioning on $\tau_s$ and using \eqref{D1} with exponent $2p$ (which holds for all deterministic times), we obtain
\[
\E|X_s-X_{\tau_s}|^{2p}
=\E\Big[\E\big(|X_s-X_{\tau_s}|^{2p}\mid \tau_s\big)\Big]
\lesssim \E|s-\tau_s|^{p}.
\]
Moreover, using $|s-\tau_s|\le |s-\cN_s^\eps|$ and Lemma~\ref{lem_y} (and $s\le t$), we have
\[
\E|s-\tau_s|^{p}\lesssim \E|s-\cN_s^\eps|^{p}\lesssim \eps^{\frac p2}.
\]
Therefore,
\[
\E|\sI_1(t)|^{2p}
\lesssim \int_0^t \eps^{\frac p2}\,\dif s
\lesssim \eps^{\frac p2}.
\]

\medskip
\noindent\textbf{Estimate of $\sI_3$.}
Write the integral with respect to $W_{\cN^\eps}$ in terms of the compensated Poisson random measure
$\widetilde{\cH}^\eps$ associated with the jumps of $W_{\cN^\eps}$:
\[
\sI_3(t)
=\int_0^t\int_{\R^m}\big(\sigma(s,X_{s\wedge\cN_{s-}^\eps})
-\sigma(s,X_{s-}^\eps)\big)z\,\widetilde{\cH}^\eps(\dif s,\dif z).
\]
By BDG (for exponent $2p$), \eqref{as_sigma}, and the standard estimate for compensated Poisson integrals,
\begin{align*}
\E|\sI_3(t)|^{2p}
&\lesssim\E\Big(\int_0^t\int_{\R^m}
\big|(\sigma(s,X_{s\wedge\cN_s^\eps})-\sigma(s,X_s^\eps))z\big|^2
\frac{\nu_\eps(\dif z)}{\eps}\,\dif s\Big)^{p} \\
&\lesssim\E\Big(\int_0^t |X_{s\wedge\cN_s^\eps}-X_s^\eps|^{2}
\Big(\int_{\R^m}|z|^2\frac{\nu_\eps(\dif z)}{\eps}\Big)\dif s\Big)^{p}.
\end{align*}
Since $\nu_\eps\sim N(0,\eps I_m)$, one has $\int_{\R^m}|z|^2\nu_\eps(\dif z)=m\eps$, hence
$\int_{\R^m}|z|^2\frac{\nu_\eps(\dif z)}{\eps}=m$ and thus
\[
\E|\sI_3(t)|^{2p}
\lesssim\E\Big(\int_0^t |X_{s\wedge\cN_s^\eps}-X_s^\eps|^{2}\,\dif s\Big)^{p}.
\]
Again using $(\int_0^t f)^p\le t^{p-1}\int_0^t f^p$,
\[
\E|\sI_3(t)|^{2p}
\lesssim\int_0^t \E|X_{s\wedge\cN_s^\eps}-X_s^\eps|^{2p}\,\dif s.
\]
Finally,
\begin{align*}
\E|X_{s\wedge\cN_s^\eps}-X_s^\eps|^{2p}
&\lesssim\E|X_{s\wedge\cN_s^\eps}-X_s|^{2p}+\E|X_s-X_s^\eps|^{2p}\\
&\lesssim \eps^{\frac p2}+\E|X_s-X_s^\eps|^{2p},
\end{align*}
where we used the bound $\E|X_s-X_{s\wedge\cN_{s-}^\eps}|^{2p}\lesssim \eps^{p/2}$
proved in the estimate of $\sI_1$.
Integrating over $s\in[0,t]$ yields
\[
\E|\sI_3(t)|^{2p}
\lesssim\eps^{\frac p2}+\int_0^t \E|X_s-X_s^\eps|^{2p}\,\dif s.
\]

\medskip
Combining the two estimates completes the proof.
\end{proof}

\medskip

We now turn to the most delicate term $\sI_2$.

\begin{lemma}\label{I_3}
Under {\bf(H$^p_0$)} and {\bf(H$^t_\sigma$)}, for all $T>0$, there exists a constant $C>0$ such that for all $\eps\in(0,1)$ and $t\in[0,T]$,
\[
\E|\sI_2(t)|^{2p}\le C\,\eps^{\frac{\beta p}{2}}.
\]
\end{lemma}

\begin{proof}
Fix $t>0$, and let $K:=\lfloor t/\eps\rfloor$.

\medskip
\noindent\textbf{Step 1: a convenient representation of the $W_{\cN^\eps}$--integral.}
Using that $W_{\cN^\eps}$ jumps at the Poisson times $S_k^\eps$ and
$\Delta W_{\cN^\eps_{S_k^\eps}}=W_{k\eps}-W_{(k-1)\eps}$, we have
\begin{align*} 
\int_0^t \sigma(s,X_{s\wedge\cN_{s-}^\eps})\,\dif W_{\cN_s^\eps}
&=\sum_{0<s\le t}\sigma(s,X_{s\wedge\cN_{s-}^\eps})\,\Delta W_{\cN_s^\eps}\\
&=\sum_{k=1}^{\cN_{t/\eps}}\sigma(S_k^\eps,X_{S_k^\eps\wedge (k-1)\eps})\,(W_{k\eps}-W_{(k-1)\eps}),
\end{align*}
and
\[
\int_0^t \sigma(s,X_{s\wedge\cN_{s-}^\eps})\,\dif W_s
=\sum_{k=1}^{K}\int_{(k-1)\eps}^{k\eps}\sigma(s,X_{s\wedge\cN_{s-}^\eps})\,\dif W_s
+\int_{K\eps}^{t}\sigma(s,X_{s\wedge\cN_{s-}^\eps})\,\dif W_s.
\]
Therefore, writing $\sI_2(t)=\int_0^t \sigma(\cdot)\dif W-\int_0^t \sigma(\cdot)\dif W_{\cN^\eps}$,
we decompose
\[
\sI_2(t)=\sI_{21}(t)-\sI_{22}(t)+\sI_{23}(t),
\]
where
\begin{align*}
\sI_{21}(t)&:=\sum_{k=1}^{K}\int_{(k-1)\eps}^{k\eps}
\Big(\sigma(s,X_{s\wedge\cN_{s-}^\eps})-\sigma(S_k^\eps,X_{S_k^\eps\wedge (k-1)\eps})\Big)\,\dif W_s,\\
\sI_{22}(t)&:=\sum_{k=K+1}^{\cN_{t/\eps}}
\sigma(S_k^\eps,X_{S_k^\eps\wedge (k-1)\eps})\,(W_{k\eps}-W_{(k-1)\eps}),\\
\sI_{23}(t)&:=\int_{K\eps}^t\sigma(s,X_{s\wedge\cN_{s-}^\eps})\,\dif W_s .
\end{align*}
(As usual, if the upper index is smaller than the lower one, the sum is understood with the sign convention
$\sum_{k=k_1}^{k_2}a_k:=-\sum_{k=k_2+1}^{k_1-1}a_k$.)

\medskip
\noindent\textbf{Step 2: estimate of $\E|\sI_{21}(t)|^{2p}$.}
By the BDG inequality for continuous martingales,
\[
\E|\sI_{21}(t)|^{2p}
\lesssim
\E\Big(\sum_{k=1}^{K}\int_{(k-1)\eps}^{k\eps}\delta_k^\eps(s)\,\dif s\Big)^p,
\]
where
\[
\delta_k^\eps(s):=\big\|\sigma(s,X_{s\wedge\cN_s^\eps})-\sigma(S_k^\eps,X_{S_k^\eps\wedge (k-1)\eps})\big\|_{\rm HS}^2.
\]
Since $p\ge1$, Jensen's inequality gives
\[
\E|\sI_{21}(t)|^{2p}
\lesssim
\Big(\sum_{k=1}^{K}\int_{(k-1)\eps}^{k\eps}\big(\E[\delta_k^\eps(s)^p]\big)^{1/p}\,\dif s\Big)^p.
\]

Next, by {\bf(H$^t_\sigma$)} and the spatial Lipschitz bound in {\bf(H$^p_0$)},
\[
\delta_k^\eps(s)
\lesssim |s^\alpha-(S_k^\eps)^\alpha|^\beta\,(1+|X_{s\wedge\cN_s^\eps}|^2)
+|X_{s\wedge\cN_s^\eps}-X_{S_k^\eps\wedge (k-1)\eps}|^2.
\]
Therefore, using $(a+b)^p\le 2^{p-1}(a^p+b^p)$,
\begin{align*}
\E[\delta_k^\eps(s)^p]
&\lesssim \E\Big(|s^\alpha-(S_k^\eps)^\alpha|^{\beta p}(1+|X_{s\wedge\cN_s^\eps}|^2)^p\Big)
+\E|X_{s\wedge\cN_s^\eps}-X_{S_k^\eps\wedge (k-1)\eps}|^{2p}.
\end{align*}
Using \eqref{X_bound}, we have $\sup_{r\le t}\E(1+|X_r|^2)^p<\infty$, hence
\[
\E\Big(|s^\alpha-(S_k^\eps)^\alpha|^{\beta p}(1+|X_{s\wedge\cN_s^\eps}|^2)^p\Big)
\lesssim \E|s^\alpha-(S_k^\eps)^\alpha|^{\beta p}.
\]
For $s\in((k-1)\eps,k\eps]$, Lemma~\ref{lem_x} yields
\[
\E|s^\alpha-(S_k^\eps)^\alpha|^{\beta p}
\lesssim (k\eps)^{\alpha\beta p}\,k^{-\frac{\beta p}{2}} .
\]
For the space increment term, conditioning on the random times and using \eqref{D1} together with the independence of $X$ and $\cN^\eps$, we get
\[
\E|X_{s\wedge\cN_s^\eps}-X_{S_k^\eps\wedge (k-1)\eps}|^{2p}
\lesssim \E\big|\,s\wedge\cN_s^\eps-S_k^\eps\wedge (k-1)\eps\,\big|^{p}.
\]
Moreover, by $(a+b+c)^p\lesssim a^p+b^p+c^p$, Lemma~\ref{lem_x}, and Lemma~\ref{lem_y},
\[
\E\big|s\wedge\cN_s^\eps-S_k^\eps\wedge (k-1)\eps\big|^{p}
\lesssim \E|s-S_k^\eps|^{p}+\E|\cN_s^\eps-s|^{p}+\eps^{p}
\lesssim \eps^{\frac p2}.
\]
Combining the above, for $s\in((k-1)\eps,k\eps]$,
\[
\E[\delta_k^\eps(s)^p]\lesssim \eps^{\alpha\beta p}k^{\alpha\beta p-\frac{\beta p}{2}}+\eps^{\frac p2}.
\]
Therefore,
\begin{align*}
\big(\E|\sI_{21}(t)|^{2p}\big)^{1/p}
&\lesssim\sum_{k=1}^{K}\int_{(k-1)\eps}^{k\eps}
\Big(\eps^{\alpha\beta p}k^{\alpha\beta p-\frac{\beta p}{2}}+\eps^{\frac p2}\Big)^{\frac1p}\,\dif s\\
&\lesssim
\eps^{\alpha\beta+1}\sum_{k=1}^{K}k^{\alpha\beta-\beta/2}
+\eps^{\frac12+1}\sum_{k=1}^{K}1
\lesssim \eps^{\frac{\beta}{2}}+\eps^{\frac12}.
\end{align*}
In particular, since $\beta\in(0,1]$ and $\eps\in(0,1)$,
\[
\E|\sI_{21}(t)|^{2p}\lesssim\eps^{\frac{\beta p}{2}}.
\]
\medskip
\noindent\textbf{Step 3: estimate of $\E|\sI_{22}(t)|^{2p}$.}
Set $\Delta W_k:=W_{k\eps}-W_{(k-1)\eps}$ and
\[
\xi_k^\eps:=\sigma(S_k^\eps,X_{S_k^\eps\wedge (k-1)\eps}).
\]
Conditionally on $\cN^\eps$, the random sum
\[
\sum_{k=K+1}^{\cN_{t/\eps}}\xi_k^\eps\,\Delta W_k
\]
is a discrete-time martingale with respect to the Brownian filtration, because $\xi_k^\eps$ is
$\sigma(W_u:u\le (k-1)\eps,\cN^\eps)$-measurable and $\Delta W_k$ is independent of the past.
Hence, by the discrete BDG inequality,
\[
\E\big(|\sI_{22}(t)|^{2p}\mid \cN^\eps\big)
\lesssim
\E\left(\Big(\sum_{k=K+1}^{\cN_{t/\eps}}|\xi_k^\eps\Delta W_k|^2\Big)^p\Bigm|\cN^\eps\right).
\]
Using $(\sum_{j=1}^n a_j)^p\le n^{p-1}\sum_{j=1}^n a_j^p$ for $a_j\ge0$ and \eqref{as_sigma}, \eqref{X_bound},
\begin{align*}
\E\big(|\sI_{22}(t)|^{2p}\mid \cN^\eps\big)
&\lesssim
(\cN_{t/\eps}-K)^{p-1}\sum_{k=K+1}^{\cN_{t/\eps}}
\E\big(|\xi_k^\eps|^{2p}|\Delta W_k|^{2p}\mid \cN^\eps\big)\\
&\lesssim
\eps^p\,(\cN_{t/\eps}-K)^{p}.
\end{align*}
Taking expectations and using Lemma~\ref{lem_y} for $\cN_{t/\eps}$,
\[
\E|\sI_{22}(t)|^{2p}
\lesssim \eps^p\,\E\big|\cN_{t/\eps}-K\big|^p
\lesssim \eps^p\Big(\E|\cN_{t/\eps}-t/\eps|^p+1\Big)
\lesssim \eps^{p/2}.
\]
\medskip
\noindent\textbf{Step 4: estimate of $\E|\sI_{23}(t)|^{2p}$.}
By BDG, \eqref{as_sigma}, \eqref{X_bound}, and $t-K\eps\le \eps$,
\begin{align*}
\E|\sI_{23}(t)|^{2p}
&\lesssim\E\Big(\int_{K\eps}^{t}\|\sigma(s,X_{s\wedge\cN_{s-}^\eps})\|_{\rm HS}^2\,\dif s\Big)^p
\lesssim\eps^{p-1}\int_{K\eps}^{t}\E\|\sigma(s,X_{s\wedge\cN_{s-}^\eps})\|_{\rm HS}^{2p}\,\dif s
\lesssim \eps^{p}.
\end{align*}

\medskip
\noindent\textbf{Conclusion.}
By $(a+b+c)^{2p}\le 3^{2p-1}(a^{2p}+b^{2p}+c^{2p})$,
\[
\E|\sI_2(t)|^{2p}
\lesssim \E|\sI_{21}(t)|^{2p}+\E|\sI_{22}(t)|^{2p}+\E|\sI_{23}(t)|^{2p}
\lesssim \eps^{\frac{\beta p}{2}}.
\]
This proves the desired estimate.
\end{proof}

\medskip

Finally, we estimate the drift error term $\sI_4$.

\begin{lemma}\label{I_4}
Under {\bf(H$^p_0$)}, for all $T>0$, there exists a constant $C>0$ such that for all $\eps\in(0,1)$ and $t\in[0,T]$,
\[
\E|\sI_4(t)|^{2p}
\le C\Big(\eps^{p}+\int_0^t\E|X_s-X_s^\eps|^{2p}\,\dif s\Big).
\]
\end{lemma}

\begin{proof}
Recalling $\widetilde{\cN}_t^\eps:=\cN_t^\eps-t$, we may write
\[
\sI_4(t)
=\int_0^t b(s,X_{s-}^\eps)\,\dif\widetilde{\cN}_s^\eps
+\int_0^t\big(b(s,X_s^\eps)-b(s,X_s)\big)\,\dif s=:J_1(t)+J_2(t).
\]

\medskip
\noindent\textbf{Estimate of $J_1$.}
Recall that $\cN_t^\eps=\eps\,\cN_{t/\eps}$ and $\widetilde{\cN}_t^\eps=\cN_t^\eps-t$.
Then $\widetilde{\cN}^\eps$ is a compensated Poisson martingale with jump size $\eps$
and jump intensity $1/\eps$. Since $\widetilde{\cN}_t^\eps=\eps\,\widetilde{\cN}_{t/\eps}$, we may rewrite
\[
J_1(t)=\int_0^t b(s,X_{s-}^\eps)\,\dif(\eps\,\widetilde{\cN}_{s/\eps})
=\int_0^t \eps\, b(s,X_{s-}^\eps)\,\dif\widetilde{\cN}_{s/\eps}.
\]
By BDG (Kunita's) inequality for compensated Poisson integrals with exponent $2p\ge2$,
\begin{align*}
\mE|J_1(t)|^{2p}
\lesssim\;
\mE\Big(\int_0^t |\eps\,b(s,X_s^\eps)|^2\,\frac{\dif s}{\eps}\Big)^p
\;+\;
\mE\int_0^t |\eps\,b(s,X_s^\eps)|^{2p}\,\frac{\dif s}{\eps}.
\end{align*}
Consequently,
\begin{align*}
\mE|J_1(t)|^{2p}\lesssim
\eps^{p}\,\mE\Big(\int_0^t |b(s,X_s^\eps)|^2\,\dif s\Big)^p
+\eps^{2p-1}\int_0^t \mE|b(s,X_s^\eps)|^{2p}\,\dif s\lesssim \eps^{p},
\end{align*}
where we used {\bf (H$^p_0$)} and Lemma~\ref{bound} to bound the integral terms.

\medskip
\noindent\textbf{Estimate of $J_2$.}
By \eqref{as_lwm_rewrite},
\[
|b(s,X_s^\eps)-b(s,X_s)|\le \ell_b(s)|X_s^\eps-X_s|.
\]
Therefore, by Cauchy--Schwarz,
\[
|J_2(t)|
\le \int_0^t \ell_b(s)|X_s^\eps-X_s|\,\dif s
\le \Big(\int_0^t \ell_b(s)^2\,\dif s\Big)^{1/2}\Big(\int_0^t |X_s^\eps-X_s|^2\,\dif s\Big)^{1/2}.
\]
Raising to power $2p$,
\[
|J_2(t)|^{2p}
\le \Big(\int_0^t \ell_b(s)^2\,\dif s\Big)^{p}\Big(\int_0^t |X_s^\eps-X_s|^2\,\dif s\Big)^{p}.
\]
Taking expectations and using $(\int_0^t f)^{p}\le t^{p-1}\int_0^t f^{p}$ with $f(s)=|X_s^\eps-X_s|^2$,
\[
\E|J_2(t)|^{2p}
\le \Big(\int_0^t \ell_b(s)^2\,\dif s\Big)^{p}\,
t^{p-1}\int_0^t \E|X_s^\eps-X_s|^{2p}\,\dif s
\lesssim_{p,t}\int_0^t \E|X_s^\eps-X_s|^{2p}\,\dif s.
\]

\medskip
\noindent\textbf{Conclusion.}
By $(a+b)^{2p}\le 2^{2p-1}(a^{2p}+b^{2p})$,
\[
\E|\sI_4(t)|^{2p}
\lesssim \E|J_1(t)|^{2p}+\E|J_2(t)|^{2p}
\lesssim \eps^{p}+\int_0^t\E|X_s-X_s^\eps|^{2p}\,\dif s.
\]
This proves the claim.
\end{proof}

\begin{proof}[Proof of Theorem~\ref{1}]
Collecting Lemmas~\ref{I_1}, \ref{I_3}, and \ref{I_4}, we obtain
\[
\mE|X_t-X_t^\eps|^{2p}
\le
C\int_0^t \mE|X_s-X_s^\eps|^{2p}\,\dif s
+
C\,\eps^{\frac{\beta p}{2}}.
\]
Gronwall's inequality yields
\[
\mE|X_t-X_t^\eps|^{2p}\le C\,\eps^{\frac{\beta p}{2}},
\]
which completes the proof.
\end{proof}

\section{Proof of Theorem \ref{2}}

In this section we prove Theorem~\ref{2}. 
Because the kernel depends explicitly on the terminal time $t$, the process
\[
t\longmapsto \int_0^t \sigma(t,s,Y_s)\,\dif W_s
\]
is, in general, not a martingale in $t$. Consequently, unlike in the Markovian setting, we shall work with pointwise-in-time $L^2$ estimates.

\begin{lemma}\label{lem:Volterra_moment}
Under {\bf(H$^\gamma_1$)} and {\bf(H$^\gamma_2$)}, for any $T>0$ there exists a constant $C_T>0$ such that
\begin{equation}\label{boundY_t}
\sup_{t\in[0,T]}\E|Y_t|^4
+\sup_{\eps\in(0,1)}\sup_{t\in[0,T]}\E|Y_t^\eps|^2
\le C_T\bigl(1+\E|Y_0|^4\bigr),
\end{equation}
and, for all $s,t\in[0,T]$,
\begin{equation}\label{JH2}
\E|Y_t-Y_s|^2 \le C_T\,|t-s|^\gamma .
\end{equation}
\end{lemma}

\begin{proof}
\noindent{\bf (i) Fourth moment bound for $Y$.}
Fix $t\in[0,T]$. By the elementary inequality $|x+y+z|^4\lesssim |x|^4+|y|^4+|z|^4$, BDG's inequality, and \eqref{h11}--\eqref{h12}, we obtain
\begin{align*}
\E|Y_t|^4
&\lesssim \E|Y_0|^4
 + \E\Big|\int_0^t \sigma(t,s,Y_s)\,\dif W_s\Big|^4
 + \E\Big|\int_0^t b(t,s,Y_s)\,\dif s\Big|^4 \\
&\lesssim \E|Y_0|^4
 + \E\Big(\int_0^t \|\sigma(t,s,Y_s)\|_{\rm HS}^2\,\dif s\Big)^2
 + \E\Big(\int_0^t |b(t,s,Y_s)|\,\dif s\Big)^4 \\
&\lesssim \E|Y_0|^4
 + \E\Big(\int_0^t \ell_1(t,s)\bigl(1+|Y_s|^2\bigr)\,\dif s\Big)^2
 + \E\Big(\int_0^t \ell_2(t,s)\bigl(1+|Y_s|^2\bigr)\,\dif s\Big)^2 .
\end{align*}
By Minkowski's integral inequality,
\begin{equation}\label{eq:Y_L2}
(\E|Y_t|^4)^{1/2}
\lesssim (\E|Y_0|^4)^{1/2}
+ \int_0^t \bigl(\ell_1(t,s)+\ell_2(t,s)\bigr)\bigl(1+(\E|Y_s|^4)^{1/2}\bigr)\,\dif s.
\end{equation}
Applying a Volterra-type Gr\"onwall inequality, together with \eqref{h13}, yields
\[
\sup_{t\in[0,T]}\E|Y_t|^4 \le C_T\bigl(1+\E|Y_0|^4\bigr).
\]

\medskip

\noindent{\bf (ii) Second moment bound for $Y^\eps$.}
Recalling \eqref{def_H}--\eqref{def_wt_H}, we may write
\begin{align*}
Y_t^\eps
&=Y_0+\int_0^t b(t,s,Y_s^\eps)\,\dif s
+\int_0^t\int_{\R^m}\Big(\sigma(t,s,Y_{s-}^\eps)z+\eps\, b(t,s,Y_{s-}^\eps)\Big)\,
\widetilde{\mathcal H}^\eps(\dif s,\dif z),
\end{align*}
where the drift term appears after compensating $\mathcal H^\eps$, whose intensity measure is
$
\frac{1}{\eps}\,\dif s\,\nu_\eps(\dif z).
$
Using $|x+y+z|^2\le 3(|x|^2+|y|^2+|z|^2)$, It\^o's isometry for compensated Poisson integrals, and Cauchy--Schwarz, we get
\begin{align*}
\E|Y_t^\eps|^2
&\le 3\E|Y_0|^2
 + 3\E\Big|\int_0^t b(t,s,Y_s^\eps)\,\dif s\Big|^2 \\
&\qquad
 + 3\E\Big|\int_0^t\!\!\int_{\R^m}
   \Big(\sigma(t,s,Y_{s-}^\eps)z+\eps\, b(t,s,Y_{s-}^\eps)\Big)\,
   \widetilde{\mathcal H}^\eps(\dif s,\dif z)\Big|^2 \\
&\le 3\E|Y_0|^2
 + 3t\int_0^t \E|b(t,s,Y_s^\eps)|^2\,\dif s \\
&\qquad
 + \frac{3}{\eps}\E\int_0^t\!\!\int_{\R^m}
   \Big|\sigma(t,s,Y_{s}^\eps)z+\eps\, b(t,s,Y_{s}^\eps)\Big|^2
   \nu_\eps(\dif z)\,\dif s .
\end{align*}
Using $|u+v|^2\le 2|u|^2+2|v|^2$, \eqref{h11}--\eqref{h12}, $\int_{\R^m}|z|^2\nu_\eps(\dif z)\lesssim \eps$, and $\eps\in(0,1)$, we obtain
\begin{align*}
\E|Y_t^\eps|^2
&\lesssim \E|Y_0|^2
 + \int_0^t \ell_2(t,s)\bigl(1+\E|Y_s^\eps|^2\bigr)\,\dif s \\
&\qquad
 + \int_0^t \ell_1(t,s)\bigl(1+\E|Y_s^\eps|^2\bigr)\,\dif s
 + \eps\int_0^t \ell_2(t,s)\bigl(1+\E|Y_s^\eps|^2\bigr)\,\dif s \\
&\lesssim \E|Y_0|^2
 + \int_0^t \bigl(\ell_1(t,s)+\ell_2(t,s)\bigr)\bigl(1+\E|Y_s^\eps|^2\bigr)\,\dif s.
\end{align*}
That is,
\begin{equation}\label{eq:Yeps_L2}
\E|Y_t^\eps|^2
\lesssim \E|Y_0|^2
+ \int_0^t \bigl(\ell_1(t,s)+\ell_2(t,s)\bigr)\bigl(1+\E|Y_s^\eps|^2\bigr)\,\dif s.
\end{equation}
Applying again a Volterra-type Gr\"onwall inequality and using \eqref{h13}, we deduce
\[
\sup_{\eps\in(0,1)}\sup_{t\in[0,T]}\E|Y_t^\eps|^2
\le C_T\bigl(1+\E|Y_0|^2\bigr)
\le C_T\bigl(1+\E|Y_0|^4\bigr).
\]
Combining this with the estimate obtained in part {\bf (i)} gives \eqref{boundY_t}.

\medskip

\noindent{\bf (iii) Time H\"older estimate for $Y$.}
Let $0\le s<t\le T$. Decompose
\begin{align*}
Y_t-Y_s
&=\int_s^t \sigma(t,r,Y_r)\,\dif W_r + \int_s^t b(t,r,Y_r)\,\dif r \\
&\quad +\int_0^s\bigl(\sigma(t,r,Y_r)-\sigma(s,r,Y_r)\bigr)\,\dif W_r
+\int_0^s\bigl(b(t,r,Y_r)-b(s,r,Y_r)\bigr)\,\dif r \\
&=: I_1+I_2+I_3+I_4.
\end{align*}
By It\^o's isometry, Cauchy--Schwarz, \eqref{h11}--\eqref{h12}, and \eqref{boundY_t},
\begin{align*}
\E|I_1|^2
&= \E\int_s^t \|\sigma(t,r,Y_r)\|_{\rm HS}^2\,\dif r
\lesssim \int_s^t \ell_1(t,r)\bigl(1+\E|Y_r|^2\bigr)\,\dif r,\\
\E|I_2|^2
&\le (t-s)\int_s^t \E|b(t,r,Y_r)|^2\,\dif r
\lesssim (t-s)\int_s^t \ell_2(t,r)\bigl(1+\E|Y_r|^2\bigr)\,\dif r.
\end{align*}
Since $t-s\le T$, the factor $(t-s)$ can be absorbed into the constant, and thus, by \eqref{h13},
\[
\E|I_1|^2+\E|I_2|^2
\lesssim \int_s^t \bigl(\ell_1(t,r)+\ell_2(t,r)\bigr)\bigl(1+\E|Y_r|^2\bigr)\,\dif r
\lesssim |t-s|^\gamma .
\]
Similarly, using \eqref{h21}--\eqref{h22} and \eqref{boundY_t},
\begin{align*}
\E|I_3|^2
&= \E\int_0^s \|\sigma(t,r,Y_r)-\sigma(s,r,Y_r)\|_{\rm HS}^2\,\dif r \\
&\lesssim \int_0^s \ell_3(t,s,r)\bigl(1+\E|Y_r|^2\bigr)\,\dif r,\\
\E|I_4|^2
&\le s\int_0^s \E|b(t,r,Y_r)-b(s,r,Y_r)|^2\,\dif r \\
&\lesssim \int_0^s \ell_4(t,s,r)\bigl(1+\E|Y_r|^2\bigr)\,\dif r.
\end{align*}
Hence, by \eqref{h23},
\[
\E|I_3|^2+\E|I_4|^2 \lesssim |t-s|^\gamma.
\]
Combining the four estimates, we conclude \eqref{JH2}.
\end{proof}

\smallskip

To establish strong convergence, we decompose the error as
\begin{align*}
Y_t-Y_t^\eps
=&\int_0^t\big(\sigma(t,s,Y_s)-\sigma(t,s,Y_{s\wedge \cN_{s-}^\eps})\big)\,\dif W_s\\
&+\int_0^t \sigma(t,s,Y_{s\wedge \cN_{s-}^\eps})\,\dif\big(W_s-W_{\cN_s^\eps}\big)\\
&+\int_0^t\big(\sigma(t,s,Y_{s\wedge \cN_{s-}^\eps})-\sigma(t,s,Y_{s-}^\eps)\big)\,\dif W_{\cN_s^\eps}\\
&+\Big(\int_0^t b(t,s,Y_{s-}^\eps)\,\dif \cN_s^\eps-\int_0^t b(t,s,Y_s)\,\dif s\Big)\\
=:&\sJ_1(t)+\sJ_2(t)+\sJ_3(t)+\sJ_4(t).
\end{align*}
In what follows, we estimate $\sJ_1(t)$, $\sJ_2(t)$, $\sJ_3(t)$, and $\sJ_4(t)$ separately.
We emphasize that, for Volterra equations, the estimate \eqref{es_x} does not admit a direct analogue. Consequently, the bound on $\sJ_2(t)$ requires a different discretization argument from that used for $\sI_2(t)$ in the proof of Theorem~\ref{1}.

\begin{lemma}\label{lem:J13}
Assume {\bf(H$^\gamma_1$)} and {\bf(H$^\gamma_2$)}. Then  
for all $T>0$, 
there exists $C>0$ such that for all $\eps\in(0,1)$, and $t\in[0,T]$,
\[
\E|\sJ_1(t)|^2+\E|\sJ_3(t)|^2
\le C\Big(\int_0^t \ell_1(t,s)\E|Y_s-Y_s^\eps|^2\,\dif s+\eps^{\gamma/2}\Big).
\]
\end{lemma}

\begin{proof}
\noindent{\bf Estimate of $\sJ_1(t)$.}
By It\^o's isometry and \eqref{h11},
\[
\E|\sJ_1(t)|^2
\lesssim \int_0^t \ell_1(t,s)\,\E|Y_s-Y_{s\wedge\cN_{s-}^\eps}|^2\,\dif s.
\]
Since $\cN^\eps$ is independent of $Y$, by conditioning on $\tau:=s\wedge\cN_{s-}^\eps$ and using \eqref{JH2},
\[
\E|Y_s-Y_\tau|^2
=\E\big[\E(|Y_s-Y_\tau|^2\mid \tau)\big]
\lesssim \E|s-\tau|^\gamma
\le \E|s-\cN_s^\eps|^\gamma+\eps^\gamma
\lesssim \eps^{\gamma/2},
\]
where we used Lemma~\ref{lem_y} with exponent $\gamma$ and $s\le T$.
Hence 
$$
\E|\sJ_1(t)|^2\lesssim \eps^{\gamma/2}\int_0^t\ell_1(t,s)\,\dif s\lesssim \eps^{\gamma/2}.
$$

\medskip
\noindent{\bf Estimate of $\sJ_3(t)$.}
Using the compensated Poisson measure $\widetilde{\cH}^\eps$,
\[
\sJ_3(t)=\int_0^t\int_{\R^m}\big(\sigma(t,s,Y_{s\wedge\cN_{s-}^\eps})-\sigma(t,s,Y_{s-}^\eps)\big)z\,\widetilde{\cH}^\eps(\dif s,\dif z).
\]
By It\^o's isometry and \eqref{h11},
\begin{align*}
\E|\sJ_3(t)|^2
&=\E\int_0^t\int_{\R^m}\big\|\sigma(t,s,Y_{s\wedge\cN_{s-}^\eps})-\sigma(t,s,Y_{s-}^\eps)\big\|_{\rm HS}^2
|z|^2\,\frac{\nu_\eps(\dif z)}{\eps}\,\dif s\\
&\lesssim \int_0^t \ell_1(t,s)\,\E|Y_{s\wedge\cN_{s-}^\eps}-Y_{s}^\eps|^2\,\dif s\\
&\le C\int_0^t \ell_1(t,s)\E|Y_s-Y_s^\eps|^2\,\dif s
+C\int_0^t \ell_1(t,s)\E|Y_s-Y_{s\wedge\cN_{s-}^\eps}|^2\,\dif s\\
&\lesssim \int_0^t \ell_1(t,s)\E|Y_s-Y_s^\eps|^2\,\dif s+\eps^{\gamma/2},
\end{align*}
where the last step uses the bound on $\E|Y_s-Y_{s\wedge\cN_{s-}^\eps}|^2$ proved above.
\end{proof}

\begin{lemma}\label{lem:J2}
Under {\bf(H$^\gamma_1$)}, {\bf(H$^\gamma_2$)} and {\bf(H$^\gamma_3$)}, 
for all $T>0$,  there exists a constant $C>0$ such that for all $\eps\in(0,1)$ and $t\in[0,T]$,
\[
\E|\sJ_2(t)|^2 \le C\,\eps^{\gamma/(2(2+\gamma))}.
\]
\end{lemma}

\begin{proof}
Fix $\delta\in(0,1\wedge t/3)$. We use the decomposition (recall $s_\delta:=[s/\delta]\delta$  
and $t_\delta=[t/\delta]\delta$)
\begin{align*}
\sJ_2(t)
&=\int_0^\delta \sigma(t,s,Y_{s\wedge\cN_{s-}^\eps})\,\dif\big(W_s-W_{\cN_s^\eps}\big)
+\int_{t_\delta}^t \sigma(t,s,Y_{s\wedge\cN_{s-}^\eps})\,\dif\big(W_s-W_{\cN_s^\eps}\big)\\
&\quad+\int_\delta^{t_\delta}\Big[\sigma(t,s,Y_{s\wedge\cN_{s-}^\eps})-\sigma(t,s_\delta,Y_{s_\delta\wedge\cN_{s_\delta-}^\eps})\Big]\dif\big(W_s-W_{\cN_s^\eps}\big)\\
&\quad+\int_\delta^{t_\delta}\sigma(t,s_\delta,Y_{s_\delta\wedge\cN_{s_\delta-}^\eps})\,\dif\big(W_s-W_{\cN_s^\eps}\big)\\
&=:\sJ_{21}(t)+\sJ_{22}(t)+\sJ_{23}(t)+\sJ_{24}(t).
\end{align*}

\medskip\noindent
\textbf{Step 1: a basic $L^2$ bound for the driving difference.}
For any predictable $f:[0,t]\to\R^{d\times m}$ with $\E\int_0^t\|f(s)\|_{\rm HS}^2\,\dif s<\infty$,
\begin{equation}\label{eq:basic_diff_isom}
\E\Big|\int_0^t f(s)\,\dif\big(W_s-W_{\cN_s^\eps}\big)\Big|^2
\le C\,\E\int_0^t \|f(s)\|_{\rm HS}^2\,\dif s .
\end{equation}
Indeed,
\[
\int_0^t f(s)\,\dif(W_s-W_{\cN_s^\eps})
=\int_0^t f(s)\,\dif W_s-\int_0^t f(s)\,\dif W_{\cN_s^\eps}.
\]
Hence, by $|A-B|^2\le2|A|^2+2|B|^2$ and It\^o's isometry,
\[
\E\Big|\int_0^t f(s)\,\dif W_s\Big|^2=\E\int_0^t\|f(s)\|_{\rm HS}^2\,\dif s.
\]
Moreover, writing $\dif W_{\cN_s^\eps}$ via the compensated Poisson measure $\widetilde{\mathcal H}^\eps$,
\[
\int_0^t f(s)\,\dif W_{\cN_s^\eps}
=\int_0^t\int_{\R^m} f(s)z\,\widetilde{\mathcal H}^\eps(\dif s,\dif z),
\]
and the isometry for compensated Poisson integrals yields
\[
\E\Big|\int_0^t f(s)\,\dif W_{\cN_s^\eps}\Big|^2
=\E\int_0^t\int_{\R^m}|f(s)z|^2\,\frac{\nu_\eps(\dif z)}{\eps}\,\dif s
\le C\,\E\int_0^t\|f(s)\|_{\rm HS}^2\,\dif s,
\]
since $\int_{\R^m}|z|^2\nu_\eps(\dif z)=m\eps$. This proves \eqref{eq:basic_diff_isom}.

\medskip\noindent
\textbf{Step 2: estimates for $\sJ_{21}$ and $\sJ_{22}$.}
Applying \eqref{eq:basic_diff_isom} with $f(s)=\sigma(t,s,Y_{s\wedge\cN_{s-}^\eps})\mathbf 1_{[0,\delta]}(s)$, then using \eqref{h11}, \eqref{boundY_t} and \eqref{h13},
\begin{align*}
\E|\sJ_{21}(t)|^2
&\le C\,\E\int_0^\delta \|\sigma(t,s,Y_{s\wedge\cN_s^\eps})\|_{\rm HS}^2\,\dif s\\
&\le C\int_0^\delta \ell_1(t,s)\bigl(1+\E|Y_{s\wedge\cN_s^\eps}|^2\bigr)\,\dif s
\le C\,\delta^\gamma.
\end{align*}
The same argument on $[t_\delta,t]$ gives
\[
\E|\sJ_{22}(t)|^2\le C\,\delta^\gamma.
\]

\medskip\noindent
\textbf{Step 3: estimate for $\sJ_{23}$.}
By \eqref{eq:basic_diff_isom},
\[
\E|\sJ_{23}(t)|^2
\le C\,\E\int_\delta^{t_\delta}
\big\|\sigma(t,s,Y_{s\wedge\cN_{s-}^\eps})-\sigma(t,s_\delta,Y_{s_\delta\wedge\cN_{s_\delta-}^\eps})\big\|_{\rm HS}^2\,\dif s.
\]
Split the difference into an $x$-part and an $s$-part:
\begin{align*}
&\big\|\sigma(t,s,Y_{s\wedge\cN_{s-}^\eps})-\sigma(t,s_\delta,Y_{s_\delta\wedge\cN_{s_\delta-}^\eps})\big\|_{\rm HS}^2\\
&\qquad\le 2\big\|\sigma(t,s,Y_{s\wedge\cN_{s-}^\eps})-\sigma(t,s,Y_{s_\delta\wedge\cN_{s_\delta-}^\eps})\big\|_{\rm HS}^2
+2\big\|\sigma(t,s,Y_{s_\delta\wedge\cN_{s_\delta-}^\eps})-\sigma(t,s_\delta,Y_{s_\delta\wedge\cN_{s_\delta-}^\eps})\big\|_{\rm HS}^2.
\end{align*}
By \eqref{h11} and \eqref{h31},
\begin{align*}
\E|\sJ_{23}(t)|^2
&\le C\int_\delta^{t_\delta}\ell_1(t,s)\E|Y_{s\wedge\cN_s^\eps}-Y_{s_\delta\wedge\cN_{s_\delta}^\eps}|^2\,\dif s
+ C\int_\delta^{t_\delta}\ell_5(t,s,s_\delta)\bigl(1+\E|Y_{s_\delta\wedge\cN_{s_\delta}^\eps}|^2\bigr)\,\dif s.
\end{align*}
For the first integral, using \eqref{JH2} with conditional expectation (since $\cN^\eps$ is independent of $Y$),
\[
\E|Y_{s\wedge\cN_s^\eps}-Y_{s_\delta\wedge\cN_{s_\delta}^\eps}|^2
\lesssim \E|\,s\wedge\cN_s^\eps-s_\delta\wedge\cN_{s_\delta}^\eps\,|^\gamma.
\]
Moreover,
\[
|s\wedge\cN_s^\eps-s_\delta\wedge\cN_{s_\delta}^\eps|
\le |s-s_\delta|+|\cN_s^\eps-\cN_{s_\delta}^\eps|+|\cN_s^\eps-s|+|\cN_{s_\delta}^\eps-s_\delta|.
\]
Hence, for $\gamma\in(0,1]$ and $s\in[k\delta,(k+1)\delta)$,
\[
\E|\,s\wedge\cN_s^\eps-s_\delta\wedge\cN_{s_\delta}^\eps\,|^\gamma
\lesssim \delta^\gamma + \E|\cN_s^\eps-\cN_{s_\delta}^\eps|^\gamma + \E|\cN_s^\eps-s|^\gamma+\E|\cN_{s_\delta}^\eps-s_\delta|^\gamma.
\]
Since $\cN^\eps$ has increments with mean $s-s_\delta\le\delta$, Jensen gives
$\E|\cN_s^\eps-\cN_{s_\delta}^\eps|^\gamma\le (\E|\cN_s^\eps-\cN_{s_\delta}^\eps|)^\gamma\lesssim \delta^\gamma$,
and Lemma \ref{lem_y} (with $p=\gamma$) yields
$\E|\cN_u^\eps-u|^\gamma\lesssim\eps^{\gamma/2}$ uniformly for $u\le T$.
Therefore,
\[
\E|\,s\wedge\cN_s^\eps-s_\delta\wedge\cN_{s_\delta}^\eps\,|^\gamma
\lesssim \delta^\gamma+\eps^{\gamma/2}.
\]
Using \eqref{h13} and \eqref{h32} (together with \eqref{boundY_t}) we conclude
\[
\E|\sJ_{23}(t)|^2 \le C(\delta^\gamma+\eps^{\gamma/2}).
\]

\medskip\noindent
\textbf{Step 4: estimate for $\sJ_{24}$.}
On each interval $[k\delta,(k+1)\delta]$, the integrand is constant in $s$, hence
\[
\sJ_{24}(t)
=\sum_{k=1}^{[t/\delta]-1}\sigma(t,k\delta,Y_{k\delta\wedge\cN_{k\delta-}^\eps})
\Big((W_{(k+1)\delta}-W_{k\delta})-(W_{\cN_{(k+1)\delta}^\eps}-W_{\cN_{k\delta}^\eps})\Big).
\]
For $u\le T$, note that
\[
\E|W_u-W_{\cN_u^\eps}|^4\lesssim\E|u-\cN_u^\eps|^2\stackrel{\text{Lemma }\ref{lem_y}}{\lesssim}\eps.
\]
Thus for $a>b$ in $[0,T]$,
\[
\E\Big|(W_a-W_b)-(W_{\cN_a^\eps}-W_{\cN_b^\eps})\Big|^4
\lesssim\E|W_a-W_{\cN_a^\eps}|^4+\E|W_b-W_{\cN_b^\eps}|^4
\lesssim \eps.
\]
Using Minkowskis inequality together with \eqref{h11} and \eqref{boundY_t},
\begin{align*}
\|\sJ_{24}(t)\|_{L^2(\Omega)}
&\le\sum_{k=1}^{[t/\delta]-1}
\Big\|\sigma(t,k\delta,Y_{k\delta\wedge\cN_{k\delta-}^\eps})\Big\|_{L^4}
\cdot \Big\|(W_{(k+1)\delta}-W_{k\delta})-(W_{\cN_{(k+1)\delta}^\eps}-W_{\cN_{k\delta}^\eps})\Big\|_{L^4}\\
&\lesssim \eps^{1/4}\sum_{k=1}^{[t/\delta]-1}\Big\|\sigma(t,k\delta,Y_{k\delta\wedge\cN_{k\delta-}^\eps})\Big\|_{L^4}
\lesssim \eps^{1/4}\sum_{k=1}^{[t/\delta]-1}\sqrt{\ell_1(t,k\delta)} \\
&\lesssim \frac{\eps^{1/4}}{\delta}\int_0^t \sqrt{\ell_1(t,s_\delta)}\,\dif s
\stackrel{\eqref{h130}}{\lesssim} \frac{\eps^{1/4}}{\delta},
\end{align*}
where the last step uses Cauchy--Schwarz and $\int_0^t\ell_1(t,s)\dif s\le C$ from \eqref{h13}. Therefore,
\[
\E|\sJ_{24}(t)|^2\lesssim \frac{\eps^{1/2}}{\delta^2}.
\]

\medskip\noindent

Combining Steps 2--4,
\[
\E|\sJ_2(t)|^2
\le 4\sum_{i=1}^4 \E|\sJ_{2i}(t)|^2
\lesssim \frac{\eps^{1/2}}{\delta^2}+\delta^\gamma+\eps^{\gamma/2}.
\]
Choosing $\delta=\eps^{1/(2(2+\gamma))}$ balances the first two terms and yields
\[
\E|\sJ_2(t)|^2 \lesssim \eps^{\gamma/(2(2+\gamma))}.
\]
The case $t<\delta$ follows by the same estimate with $\delta=t$ (then $\sJ_{22}=\sJ_{23}=\sJ_{24}=0$), so the bound holds for all $t\in[0,T]$.
\end{proof}

\begin{lemma}\label{lem:J4}
Under {\bf (H$^\gamma_1$)} and {\bf (H$^\gamma_2$)},  for all $T>0$,  there exists a constant $C>0$ such that for all $\eps\in(0,1)$ and $t\in[0,T]$
\[
\E|\sJ_4(t)|^2 \lesssim_C \int_0^t \E|Y_s^\eps-Y_s|^2\,\dif s + \eps .
\]
\end{lemma}

\begin{proof}
Write $\widetilde{\cN}_t^\eps:=\cN_t^\eps-t$. Then
\[
\sJ_4(t)=\int_0^t b(t,s,Y_{s-}^\eps)\,\dif\widetilde{\cN}_s^\eps
+\int_0^t\big(b(t,s,Y_{s-}^\eps)-b(t,s,Y_s)\big)\,\dif s.
\]
By It\^o's isometry for $\widetilde{\cN}^\eps$ (note that $\langle \widetilde{\cN}^\eps\rangle_t=\eps t$) and \eqref{h12},
\[
\E\Big|\int_0^t b(t,s,Y_{s-}^\eps)\,\dif\widetilde{\cN}_s^\eps\Big|^2
= \eps\int_0^t \E|b(t,s,Y_s^\eps)|^2\,\dif s
\lesssim \eps\int_0^t \ell_2(t,s)\bigl(1+\E|Y_s^\eps|^2\bigr)\,\dif s
\lesssim \eps,
\]
where we used \eqref{boundY_t}.
For the deterministic integral, using \eqref{h12} and Cauchy--Schwarz,
\[
\E\Big|\int_0^t \big(b(t,s,Y_{s-}^\eps)-b(t,s,Y_s)\big)\,\dif s\Big|^2
\le \Big(\int_0^t \ell_2(t,s)\,\dif s\Big)\int_0^t \E|Y_s^\eps-Y_s|^2\,\dif s.
\]
Combining the two estimates gives the claim.
\end{proof}

\begin{proof}[Proof of Theorem \ref{2}]
Combining Lemmas~\ref{lem:J13}, \ref{lem:J2} and \ref{lem:J4}, we obtain for $t\in[0,T]$,
\[
\E|Y_t-Y_t^\eps|^2
\lesssim
\int_0^t (\ell_1(t,s)+1)\E|Y_s-Y_s^\eps|^2\,\dif s
+\eps^{\gamma/(2(2+\gamma))}.
\]
An application of a Volterra-type Gr\"onwall inequality yields the desired bound.
\end{proof}

\section{Examples: fractional Brownian motion}

Let $H\in(0,1)$. We consider the fractional Brownian motion (fBm) $(B_t^H)_{t\in\R}$,
i.e.\ the continuous centered Gaussian process with $B_0^H=0$ a.s.\ and covariance
\[
R_H(t,s):=\mE(B_t^HB_s^H)=\frac12\big(|t|^{2H}+|s|^{2H}-|t-s|^{2H}\big).
\]
When $H=\frac12$, $(B_t^H)_{t\ge0}$ coincides with standard Brownian motion.
For $H\neq \frac12$, fBm is not a semimartingale, hence classical It\^o calculus does not apply.

For $t\ge0$ and $H\in(0,\frac12)\cup(\frac12,1)$, fBm admits the Volterra representation
(see \cite{decreusefond97fBM})
\[
B_t^H=\int_0^t K_H(t,s)\,\dif W_s,\qquad \text{a.s.},
\]
where $(W_t)_{t\ge0}$ is a standard Brownian motion and
\[
K_H(t,s)
:=C_H\Big((t-s)^{H-\frac12}+s^{H-\frac12}F(t/s)\Big)\,\1_{\{0<s<t\}},
\]
with
\[
C_H=\left(\frac{2H\Gamma(\frac32-H)}{\Gamma(H+\frac12)\Gamma(2-2H)}\right)^{1/2},
\qquad
F(u)=\Big(\tfrac12-H\Big)\int_1^u (r-1)^{H-\frac32}\big(1-r^{H-\frac12}\big)\,\dif r.
\]
Here $\Gamma$ denotes the Gamma function.

\medskip
We consider the following stochastic Volterra equation driven by $W$ with kernel $K_H$:
\begin{equation}\label{fBM_SDE}
X_t=X_0+\int_0^t \sigma(s,X_s)K_H(t,s)\,\dif W_s+\int_0^t b(s,X_s)\,\dif s,
\end{equation}
where $X_0\in L^2(\Omega)$, and $\sigma:\R_+\times\R^d\to\R^{d\times m}$, $b:\R_+\times\R^d\to\R^d$
are Borel measurable.

Its compound Poisson approximation is defined by
\begin{equation}\label{fBM_Poisson}
X_t^\eps
=X_0+\int_0^t \sigma(s,X_{s-}^\eps)K_H(t,s)\,\dif W_{\cN_s^\eps}
+\int_0^t b(s,X_{s-}^\eps)\,\dif \cN_s^\eps.
\end{equation}

We verify below the assumptions of Theorem~\ref{2}, which yields the following consequence.

\begin{theorem}\label{3}
Let $H\in(0,\frac12)\cup(\frac12,1)$.
Assume that {\bf (H$^p_0$)} and {\bf (H$^t_\sigma$)} hold with $\ell_b(\cdot)\equiv \kappa_b$, $\alpha=1$,
and $\beta\in(0,1]$.
Let $X$ and $X^\eps$ be the solutions of \eqref{fBM_SDE} and \eqref{fBM_Poisson}, respectively.
Then for any $T>0$ and $X_0\in L^2(\Omega)$ there exists $C=C(T,H,\kappa_\sigma,\kappa_b,\beta,X_0)>0$
such that for all $\eps\in(0,1)$,
\[
\sup_{t\in[0,T]}\mE|X_t^\eps-X_t|^2\le C\,\eps^{\gamma/(2(2+\gamma))},
\]
where for any $\eps'>0$ we may take
\[
\gamma
:=\beta\wedge\Big(\gamma_H-\eps'\,\1_{\{H\in(0,1/3]\cup\{3/4\}\}}\Big),
\qquad
\gamma_H:=H\wedge|1-2H|\wedge(2-2H).
\]
\end{theorem}

\medskip
We note that when $H\in(0,1/3]\cup\{3/4\}$, we cannot get the exact convergence rate of\\  $\eps^{\beta\wedge\gamma_H/(2(2+\beta\wedge\gamma_H))}$, so we have to minus $\eps'$ in the definition of $\gamma$.  It is standard to check that {\bf(H$^\gamma_1$)} and {\bf(H$^\gamma_2$)} hold for the present kernel.
To verify {\bf(H$^\gamma_3$)}, we shall need the following auxiliary estimates.

\begin{lemma}\label{lem_es_int}
For any $\alpha,\beta\in[0,1)$ and $t>0$, there exists a constant $C>0$ such that for any
$0<\delta<1\wedge t/3$,
\begin{equation}\label{es_int}
\int_\delta^{t_\delta}\Big(s^{-\alpha}s_\delta^{-\beta}+(t-s)^{-\alpha}(t-s_\delta)^{-\beta}\Big)\dif s
\le C\Big(\delta^{1-\alpha-\beta}+1+\1_{\{\alpha+\beta=1\}}|\log\delta|\Big).
\end{equation}
\end{lemma}

\begin{proof}
Recall that $s_\delta\in[s-\delta,s]$ for $s\ge \delta$ and $t_\delta=[t/\delta]\delta\in[t-\delta,t]$.

\smallskip\noindent
\textbf{Step 1: estimate of $\displaystyle \int_\delta^{t_\delta}s^{-\alpha}s_\delta^{-\beta}\dif s$.}
Since $s_\delta\ge s-\delta$, we have $s_\delta^{-\beta}\le (s-\delta)^{-\beta}$ and hence
\[
\int_\delta^{t_\delta}s^{-\alpha}s_\delta^{-\beta}\dif s
\le \int_\delta^{t_\delta} s^{-\alpha}(s-\delta)^{-\beta}\dif s
= \int_\delta^{2\delta} s^{-\alpha}(s-\delta)^{-\beta}\dif s+\int_{2\delta}^{t_\delta} s^{-\alpha}(s-\delta)^{-\beta}\dif s.
\]
On $(\delta,2\delta]$, we use $s^{-\alpha}\le \delta^{-\alpha}$ to get
\[
\int_\delta^{2\delta} s^{-\alpha}(s-\delta)^{-\beta}\dif s
\le \delta^{-\alpha}\int_\delta^{2\delta} (s-\delta)^{-\beta}\dif s
\lesssim \delta^{1-\alpha-\beta}.
\]
On $[2\delta,t_\delta]$, we use $s-\delta<s$, hence
$s^{-\alpha}(s-\delta)^{-\beta}\lesssim (s-\delta)^{-(\alpha+\beta)}$, and therefore
\[
\int_{2\delta}^{t_\delta} s^{-\alpha}(s-\delta)^{-\beta}\dif s
\lesssim \int_{2\delta}^{t_\delta}(s-\delta)^{-(\alpha+\beta)}\dif s.
\]
If $\alpha+\beta\neq1$, this is $\lesssim 1+\delta^{1-\alpha-\beta}$; if $\alpha+\beta=1$, it is
$\lesssim 1+|\log\delta|$. Altogether,
\[
\int_\delta^{t_\delta}s^{-\alpha}s_\delta^{-\beta}\dif s
\lesssim \delta^{1-\alpha-\beta}+1+\1_{\{\alpha+\beta=1\}}|\log\delta|.
\]

\smallskip\noindent
\textbf{Step 2: estimate of $\displaystyle \int_\delta^{t_\delta}(t-s)^{-\alpha}(t-s_\delta)^{-\beta}\dif s$.}
We split at $t-\delta$.
For $s\in[\delta,t-\delta]$, we have $s_\delta\le s$, hence $t-s_\delta\ge t-s$ and thus
$(t-s_\delta)^{-\beta}\le (t-s)^{-\beta}$. Therefore,
\[
\int_\delta^{t-\delta}(t-s)^{-\alpha}(t-s_\delta)^{-\beta}\dif s
\le \int_\delta^{t-\delta}(t-s)^{-(\alpha+\beta)}\dif s
\lesssim 1+\delta^{1-\alpha-\beta}+\1_{\{\alpha+\beta=1\}}|\log\delta|.
\]
For $s\in[t-\delta,t_\delta]$, we use $t-s_\delta\ge t-t_\delta+\delta\ge \delta$, hence
$(t-s_\delta)^{-\beta}\le \delta^{-\beta}$ and get
\[
\int_{t-\delta}^{t_\delta}(t-s)^{-\alpha}(t-s_\delta)^{-\beta}\dif s
\le \delta^{-\beta}\int_{t-\delta}^{t_\delta}(t-s)^{-\alpha}\dif s
\lesssim \delta^{-\beta}\int_{0}^{\delta}u^{-\alpha}\dif u
\lesssim \delta^{1-\alpha-\beta}.
\]
Combining both parts yields the desired bound \eqref{es_int}.
\end{proof}

\begin{lemma}\label{H_31}
Let $H\in(0,1/2)$ and $t>0$. There exists a constant $C>0$ such that for any
$\delta<1\wedge t/3$ and any $\eps'>0$,
\begin{equation}\label{H_31_new}
\int_\delta^{t_\delta}\big|K_H(t,s)-K_H(t,s_\delta)\big|^2\,\dif s
\le C\,\delta^{(H-\eps')\wedge(1-2H)}.
\end{equation}
\end{lemma}

\begin{proof}
Recall that for $H\in(0,1/2)$,
\[
K_H(t,s)=C_H\Big((t-s)^{H-\frac12}+s^{H-\frac12}F(t/s)\Big)\1_{\{s<t\}},
\qquad s\ge0,
\]
and that $F$ is bounded on $(1,\infty)$ with $F(\infty):=\lim_{u\to\infty}F(u)\in\R$.
Throughout the proof we use $|s-s_\delta|\le \delta$ and $s_\delta\in[s-\delta,s]$.

\smallskip\noindent
\textbf{Step 1: a fractional-power difference inequality.}
Set $\alpha:=\frac12-H\in(0,1/2)$. Then $x\mapsto x^{-\alpha}$ satisfies the Hölder-type bound
\begin{equation}\label{negpow_holder}
\big|a^{-\alpha}-b^{-\alpha}\big|
\le C_\alpha |a-b|^\alpha (ab)^{-\alpha},
\qquad a,b>0,
\end{equation}
(see e.g. the standard estimate for negative fractional powers; it follows by writing
$|a^{-\alpha}-b^{-\alpha}|=a^{-\alpha}|1-(b/a)^{-\alpha}|$ and using
$|1-r^{-\alpha}|\lesssim |r-1|^\alpha$ for $r\ge1$).

Applying \eqref{negpow_holder} with $(a,b)=(t-s,t-s_\delta)$ and also $(a,b)=(s,s_\delta)$ gives
\begin{align}
\big|(t-s)^{H-\frac12}-(t-s_\delta)^{H-\frac12}\big|
&\lesssim \delta^{\frac12-H}\big((t-s)(t-s_\delta)\big)^{H-\frac12}, \label{A_term}\\
\big|s^{H-\frac12}-s_\delta^{H-\frac12}\big|
&\lesssim \delta^{\frac12-H}(ss_\delta)^{H-\frac12}. \label{B_term}
\end{align}

\smallskip\noindent
\textbf{Step 2: Hölder continuity of $F$.}
Recall $F(u)=(\frac12-H)\int_1^u (r-1)^{H-\frac32}(1-r^{H-\frac12})\,\dif r$ for $u>1$.
Let
\[
f(r):=(r-1)^{H-\frac32}\big(1-r^{H-\frac12}\big),\qquad r>1.
\]
Choose $p:=\frac{2}{2-H}\in(1,\infty)$ so that $1-\frac1p=\frac{H}{2}$.
One checks that $f\in L^p(1,\infty)$: near $r=1$, $f(r)^p\sim (r-1)^{p(H-\frac12)}$ (integrable since $p(H-\frac12)\in(-1,0)$);
as $r\to\infty$, $f(r)^p\sim (r-1)^{p(H-\frac32)}$ (also integrable since $p(H-\frac32)<-1$). 
Hence, by Hölder's inequality,
\begin{equation}\label{F_holder}
|F(u)-F(v)|
\le C \,|u-v|^{H/2},
\qquad u,v>1.
\end{equation}
Applying \eqref{F_holder} to $u=t/s$ and $v=t/s_\delta$ yields
\begin{equation}\label{F_term}
\big|F(t/s)-F(t/s_\delta)\big|
\lesssim \Big|\frac{t}{s}-\frac{t}{s_\delta}\Big|^{H/2}
\lesssim \delta^{H/2}(ss_\delta)^{-H/2}.
\end{equation}

\smallskip\noindent
\textbf{Step 3: pointwise bound for $|K_H(t,s)-K_H(t,s_\delta)|$.}
Using the definition of $K_H$, the boundedness of $F$, and \eqref{A_term}--\eqref{F_term},
we obtain
\begin{align*}
|K_H(t,s)-K_H(t,s_\delta)|
&\lesssim \big|(t-s)^{H-\frac12}-(t-s_\delta)^{H-\frac12}\big|
 +\big|s^{H-\frac12}-s_\delta^{H-\frac12}\big|\cdot |F(t/s)| \\
&\quad + s_\delta^{H-\frac12}\big|F(t/s)-F(t/s_\delta)\big| \\
&\lesssim \delta^{\frac12-H}\Big(\big((t-s)(t-s_\delta)\big)^{H-\frac12}
+(ss_\delta)^{H-\frac12}\Big)
+\delta^{H/2}\, s^{-H/2}s_\delta^{H/2-\frac12}.
\end{align*}
Squaring and using $(x+y+z)^2\lesssim x^2+y^2+z^2$ gives
\begin{equation}\label{K_diff_sq}
|K_H(t,s)-K_H(t,s_\delta)|^2
\lesssim
\delta^{1-2H}\Big((t-s)^{2H-1}(t-s_\delta)^{2H-1}+s^{2H-1}s_\delta^{2H-1}\Big)
+\delta^{H}s^{-H}s_\delta^{H-1}.
\end{equation}

\smallskip\noindent
\textbf{Step 4: integration.}
Integrating \eqref{K_diff_sq} over $s\in[\delta,t_\delta]$ and applying Lemma \ref{lem_es_int}
(with $(\alpha,\beta)=(1-2H,1-2H)$ and $(\alpha,\beta)=(H,1-H)$), we obtain
\begin{align*}
\int_\delta^{t_\delta}|K_H(t,s)-K_H(t,s_\delta)|^2\,\dif s
&\lesssim
\delta^{1-2H}\Big(\delta^{4H-1}+1+\1_{\{H=1/4\}}|\log\delta|\Big)
+\delta^{H}\Big(1+|\log\delta|\Big) \\
&\lesssim \delta^{1-2H} + \delta^{H}\big(1+|\log\delta|\big).
\end{align*}
Finally, since $|\log\delta|\lesssim \delta^{-\eps'}$ for any $\eps'>0$ and all $\delta\in(0,1)$,
we have $\delta^{H}|\log\delta|\lesssim \delta^{H-\eps'}$. Hence
\[
\int_\delta^{t_\delta}|K_H(t,s)-K_H(t,s_\delta)|^2\,\dif s
\lesssim \delta^{1-2H}+\delta^{H-\eps'}
\lesssim \delta^{(H-\eps')\wedge(1-2H)}.
\]
This proves \eqref{H_31_new}.
\end{proof}

\begin{lemma}\label{H_32}
Let $H\in(1/2,1)$ and $t>0$. There exists a constant $C>0$ such that for any
$\delta<1\wedge t/3$ and any $\eps'>0$,
\begin{equation}\label{H_32_new}
\int_\delta^{t_\delta}\big|K_H(t,s)-K_H(t,s_\delta)\big|^2\,\dif s
\le C\Big(\delta^{(2H-1)\wedge(2-2H)}+\1_{\{H=3/4\}}\delta^{1/2-\eps'}\Big).
\end{equation}
\end{lemma}

\begin{proof}
For $H\in(1/2,1)$ we use the well-known representation (obtained by a change of variables)
\begin{equation}\label{KH_rep}
K_H(t,s)=c_H\, s^{\frac12-H}\int_s^t (r-s)^{H-\frac32}\, r^{H-\frac12}\,\dif r,
\qquad 0<s<t,
\end{equation}
where $c_H:=C_H(H-\frac12)$.
Fix $0<\delta<1\wedge t/3$ and note that $s_\delta\in[s-\delta,s]$ and $|s-s_\delta|\le\delta$.
Write
\[
K_H(t,s)-K_H(t,s_\delta)=A_1(s)+A_2(s)+A_3(s),
\]
where
\begin{align*}
A_1(s)&:=c_H\, s^{\frac12-H}\int_s^t\Big((r-s)^{H-\frac32}-(r-s_\delta)^{H-\frac32}\Big) r^{H-\frac12}\,\dif r,\\
A_2(s)&:=c_H\,\Big(s^{\frac12-H}-s_\delta^{\frac12-H}\Big)\int_s^t (r-s_\delta)^{H-\frac32} r^{H-\frac12}\,\dif r,\\
A_3(s)&:=c_H\, s_\delta^{\frac12-H}\int_{s_\delta}^{s}(r-s_\delta)^{H-\frac32} r^{H-\frac12}\,\dif r.
\end{align*}

\smallskip\noindent
\textbf{Step 1: bounds for $A_2$ and $A_3$.}
Since $H>1/2$,
\[
\big|s^{\frac12-H}-s_\delta^{\frac12-H}\big|
\lesssim |s-s_\delta|^{H-\frac12}\, s_\delta^{\frac12-H}s^{\frac12-H}
\lesssim \delta^{H-\frac12}\, (ss_\delta)^{\frac12-H}.
\]
Moreover, for $r\in[s,t]$ we have $r^{H-\frac12}\le t^{H-\frac12}$ and hence
\[
\int_s^t (r-s_\delta)^{H-\frac32} r^{H-\frac12}\,\dif r
\le t^{H-\frac12}\int_s^t (r-s_\delta)^{H-\frac32}\,\dif r
\lesssim_T 1,
\]
because $H-\frac32>-1$.
Therefore
\begin{equation}\label{A2_bd}
|A_2(s)|\lesssim_T\delta^{H-\frac12}\, (ss_\delta)^{\frac12-H}.
\end{equation}
Similarly,
\[
\int_{s_\delta}^{s}(r-s_\delta)^{H-\frac32} r^{H-\frac12}\,\dif r
\le t^{H-\frac12}\int_{0}^{s-s_\delta} u^{H-\frac32}\,\dif u
\lesssim \delta^{H-\frac12},
\]
so
\begin{equation}\label{A3_bd}
|A_3(s)|\lesssim_T s_\delta^{\frac12-H}\,\delta^{H-\frac12}.
\end{equation}

\smallskip\noindent
\textbf{Step 2: bound for $A_1$.}
Set $\theta:=\frac32-H\in(1/2,1)$, so $(r-s)^{H-\frac32}=(r-s)^{-\theta}$.
We use the elementary inequality for negative powers:
\begin{equation}\label{negpow_diff}
\big|a^{-\theta}-b^{-\theta}\big|
\le C_\theta |a-b|^\theta (ab)^{-\theta},
\qquad a,b>0,
\end{equation}
which follows from $|1-u^{-\theta}|\lesssim |1-u|^\theta$ for $u\ge1$.
Applying \eqref{negpow_diff} with $a=r-s$ and $b=r-s_\delta$ yields
\[
\big|(r-s)^{H-\frac32}-(r-s_\delta)^{H-\frac32}\big|
\lesssim |s-s_\delta|^{\frac32-H}\,(r-s)^{H-\frac32}(r-s_\delta)^{H-\frac32}.
\]
Hence,
\begin{equation}\label{A1_pre}
|A_1(s)|
\lesssim |s-s_\delta|^{\frac32-H}\, s^{\frac12-H}\int_s^t (r-s)^{H-\frac32}(r-s_\delta)^{H-\frac32} r^{H-\frac12}\,\dif r.
\end{equation}
Using $r^{H-\frac12}\le t^{H-\frac12}$ and splitting the integral at $s+\delta$,
\begin{align*}
\int_s^t (r-s)^{H-\frac32}(r-s_\delta)^{H-\frac32}\,\dif r
&\le (s-s_\delta)^{H-\frac32}\int_s^{s+\delta}(r-s)^{H-\frac32}\,\dif r
   +\int_{s+\delta}^t (r-s)^{2H-3}\,\dif r\\
&\lesssim (s-s_\delta)^{H-\frac32}\delta^{H-\frac12}+\delta^{2H-2}
\lesssim (s-s_\delta)^{H-\frac32}\delta^{H-\frac12},
\end{align*}
since $2H-2<0$ and $\delta^{2H-2}\lesssim \delta^{H-\frac12}(s-s_\delta)^{H-\frac32}$ for $s-s_\delta\le\delta$.
Plugging this into \eqref{A1_pre} gives
\begin{equation}\label{A1_bd}
|A_1(s)|\lesssim_T \delta^{2H-1}\, s^{\frac12-H}\,(s-s_\delta)^{0}\ \lesssim_T \delta^{2H-1}\, s^{\frac12-H}.
\end{equation}

\smallskip\noindent
\textbf{Step 3: combine and integrate.}
From \eqref{A2_bd}--\eqref{A3_bd}--\eqref{A1_bd} and $(x+y+z)^2\lesssim x^2+y^2+z^2$, we obtain
\begin{align}
|K_H(t,s)-K_H(t,s_\delta)|^2
&\lesssim_T
\delta^{4H-2}\, s^{1-2H}
+\delta^{2H-1}\, (ss_\delta)^{1-2H}
+\delta^{2H-1}\, s_\delta^{1-2H}\no\\
&\le
\delta^{2H-1}\Big(s^{1-2H}+s_\delta^{1-2H}+(ss_\delta)^{1-2H}\Big).\label{KH32_pointwise}
\end{align}
Integrating \eqref{KH32_pointwise} over $s\in[\delta,t_\delta]$ and using Lemma \ref{es_int}
(with $(\alpha,\beta)=(0,2H-1)$ and $(\alpha,\beta)=(2H-1,0)$) yields
\[
\int_\delta^{t_\delta}|K_H(t,s)-K_H(t,s_\delta)|^2\,\dif s
\lesssim_T \delta^{2H-1}\Big(1+\delta^{3-4H}+\1_{\{H=3/4\}}|\log\delta|\Big).
\]
Therefore,
\[
\int_\delta^{t_\delta}|K_H(t,s)-K_H(t,s_\delta)|^2\,\dif s
\lesssim
\delta^{(2H-1)\wedge(2-2H)}+\1_{\{H=3/4\}}\delta^{1/2}|\log\delta|.
\]
Finally, for any $\eps'>0$ and $\delta\in(0,1)$ we have $|\log\delta|\lesssim \delta^{-\eps'}$, hence
$\delta^{1/2}|\log\delta|\lesssim \delta^{1/2-\eps'}$.
This proves \eqref{H_32_new}.
\end{proof}

\begin{proof}[Proof of Theorem \ref{3}]
By Theorem \ref{2}, it suffices to verify {\bf(H$^\gamma_1$)}, {\bf(H$^\gamma_2$)} and {\bf(H$^\gamma_3$)}.

\smallskip\noindent
\textbf{Step 1: verification of {\bf(H$^\gamma_1$)}.}
By {\bf(H$^p_0$)} we may take
\[
\ell_1(t,s)=\kappa_\sigma^2 K_H(t,s)^2,
\qquad
\ell_2(t,s)=\kappa_b^2,
\]
so that \eqref{h11} and \eqref{h12} hold.
Moreover, the kernel estimate
\begin{equation}\label{es_ker_new}
K_H(t,s)\lesssim |t-s|^{H-\frac12}+s^{-|H-\frac12|}
\end{equation}
yields, for any $0\le s<t\le T$,
\[
\int_0^s K_H(t,r)^2\,\dif r+\int_{t-s}^t K_H(t,r)^2\,\dif r
\lesssim s^{2H}+s^{2-2H},
\]
and hence {\bf(H$^\gamma_1$)} holds.

\smallskip\noindent
\textbf{Step 2: verification of {\bf(H$^\gamma_2$)}.}
Again by {\bf(H$^p_0$)} we set
\[
\ell_3(t',t,s)=\kappa_\sigma^2\big(K_H(t,s)-K_H(t',s)\big)^2,
\qquad
\ell_4(t',t,s)=0,
\]
so that \eqref{h21} and \eqref{h22} hold.
For $t<t'$, It\^o's isometry gives
\[
\int_0^t\big(K_H(t,s)-K_H(t',s)\big)^2\,\dif s
=
\E\Big(\int_0^t\big(K_H(t,s)-K_H(t',s)\big)\,\dif W_s\Big)^2.
\]
Since $K_H(t,s)=0$ for $s>t$, the Volterra representation of fractional Brownian motion yields
\[
\int_0^{t'}\big(K_H(t,s)-K_H(t',s)\big)^2\,\dif s
=\E|B_t^H-B_{t'}^H|^2
=|t-t'|^{2H}.
\]
Therefore
\[
\int_0^t\big(K_H(t,s)-K_H(t',s)\big)^2\,\dif s\le |t-t'|^{2H},
\]
and thus {\bf(H$^\gamma_2$)} holds.

\smallskip\noindent
\textbf{Step 3: verification of {\bf(H$^\gamma_3$)}.}
By {\bf(H$_\sigma^t$)} {with $\alpha=1$}, for any $s,s'\in(0,t)$ and $x\in\R^d$,
\begin{align*}
&|\sigma(s,x)K_H(t,s)-\sigma(s',x)K_H(t,s')|^2\\
&\le \|\sigma(s,x)-\sigma(s',x)\|_{\HS}^2\,K_H(t,s)^2
+\|\sigma(s',x)\|_{\HS}^2\,|K_H(t,s)-K_H(t,s')|^2\\
&\le \Big(\kappa_\tau |s-s'|^\beta K_H(t,s)^2
+\kappa_\sigma^2 |K_H(t,s)-K_H(t,s')|^2\Big)(1+|x|^2)\\
&=:\ell_5(t,s,s')(1+|x|^2).
\end{align*}
Fix $t\in(0,T]$ and $0<\delta<1\wedge t/3$. Since $|s-s_\delta|\le\delta$, we have
\begin{align*}
\int_\delta^{t_\delta}\ell_5(t,s,s_\delta)\,\dif s
&\lesssim \delta^\beta \int_\delta^{t_\delta} K_H(t,s)^2\,\dif s
      + \int_\delta^{t_\delta}|K_H(t,s)-K_H(t,s_\delta)|^2\,\dif s\\
&=: J_1+J_2.
\end{align*}
For $J_1$, using \eqref{es_ker_new} and $t\le T$, we obtain
\[
\int_\delta^{t_\delta} K_H(t,s)^2\,\dif s
\lesssim \int_0^t \big(|t-s|^{2H-1}+s^{-2|H-\frac12|}\big)\,\dif s
\lesssim_T t^{2H}+t^{2-2H}
\lesssim_T t^{(2H)\wedge(2-2H)}.
\]
Hence
\[
J_1\lesssim_T \delta^\beta\, t^{(2H)\wedge(2-2H)}.
\]
For $J_2$, we use the kernel discretization estimates proved above:
\begin{itemize}
\item if $H\in(0,1/2)$, then Lemma \ref{H_31} yields
$
J_2\lesssim \delta^{(H-\eps')\wedge(1-2H)};
$
\item if $H\in(1/2,1)$, then Lemma \ref{H_32} yields
$
J_2\lesssim \delta^{(2H-1)\wedge(2-2H)}+\1_{\{H=3/4\}}\delta^{1/2-\eps'}.
$
\end{itemize}
Combining these bounds, we conclude that for any $\eps'>0$,
\[
\int_\delta^{t_\delta}\ell_5(t,s,s_\delta)\,\dif s
\lesssim_T
\delta^\beta\, t^{(2H)\wedge(2-2H)}
+
\delta^{\,H\wedge|1-2H|\wedge(2-2H)-\eps'\,\1_{\{H\in(0,1/3]\cup\{3/4\}\}}}.
\]
Therefore {\bf(H$^\gamma_3$)} holds with
\[
\gamma=\beta\wedge\Big(H\wedge|1-2H|\wedge(2-2H)-\eps'\,\1_{\{H\in(0,1/3]\cup\{3/4\}\}}\Big),
\]
and the proof is complete.
\end{proof}

\section{Numerical Experiments}\label{sec:numerics}

In this section, we report two numerical experiments that illustrate the performance of the proposed Poisson approximation scheme. In both tests, we observe that, compared with the classical Euler--Maruyama method, the Poisson scheme is noticeably more robust when the coefficients exhibit singular behavior in time (or in the Volterra kernel).
Here are two concrete algorithms for SDEs and SVEs approximated by compound Poisson processes.

\begin{algorithm}[ht]
\caption{Compound Poisson approximation for an SDE}\label{alg:sde-poisson}
\begin{algorithmic}[1]
\Require Final time $T$, parameter $\eps\in(0,1)$, initial value $X_0$, coefficients $b,\sigma$.
\State Set $S_0^\eps \gets 0$, $X_0^\eps \gets X_0$, and $k\gets 0$.
\While{$S_k^\eps<T$}
  \State Sample $T_{k+1}\sim \mathrm{Exp}(1)$ and set $S_{k+1}^\eps \gets S_k^\eps+\eps T_{k+1}$.
  \State Sample $G_{k+1}\sim \mathcal N(0,\eps I_m)$ independently.
  \State Update
  \[
  X_{k+1}^\eps \gets X_k^\eps
  + \sigma(S_{k+1}^\eps,X_k^\eps)\,G_{k+1}
  + \eps\, b(S_{k+1}^\eps,X_k^\eps).
  \]
  \State Set $k\gets k+1$.
\EndWhile
\State Define the càdlàg approximation by $X_t^\eps:=X_k^\eps$ for $t\in[S_k^\eps,S_{k+1}^\eps)$.
\end{algorithmic}
\end{algorithm}

\begin{algorithm}[ht]
\caption{Compound Poisson approximation for a stochastic Volterra equation}\label{alg:volterra-poisson}
\begin{algorithmic}[1]
\Require Final time $T$, parameter $\eps\in(0,1)$, initial value $Y_0$, coefficients $b,\sigma$, observation grid $\{t_n\}_{n=0}^N$.
\State Generate Poisson jump times $S_j^\eps=\eps\sum_{i=1}^j T_i$ with $T_i\sim \mathrm{Exp}(1)$ i.i.d., up to the last index $J$ such that $S_J^\eps\le T$.
\State Sample independent Gaussian jumps $G_j\sim\mathcal N(0,\eps I_m)$, $j=1,\dots,J$.
\State Set $Y_{S_0^\eps}^\eps\gets Y_0$.
\For{$j=1,\dots,J$}
  \State Compute the value at the physical jump time $S_j^\eps$ by
  \[
  Y_{S_j^\eps}^\eps \gets
  Y_0+\sum_{i=1}^{j}
  \Big[\sigma(S_j^\eps,S_i^\eps,Y_{S_{i-1}^\eps}^\eps)\,G_i
  +\eps\, b(S_j^\eps,S_i^\eps,Y_{S_{i-1}^\eps}^\eps)\Big].
  \]
\EndFor
\For{$n=0,\dots,N$}
  \State Let $J_n:=\max\{j: S_j^\eps\le t_n\}$.
  \State Compute
  \[
  Y_{t_n}^\eps \gets
  Y_0+\sum_{i=1}^{J_n}
  \Big[\sigma(t_n,S_i^\eps,Y_{S_{i-1}^\eps}^\eps)\,G_i
  +\eps\, b(t_n,S_i^\eps,Y_{S_{i-1}^\eps}^\eps)\Big].
  \]
\EndFor
\State Output $\{Y_{t_n}^\eps\}_{n=0}^N$.
\end{algorithmic}
\end{algorithm}

Note that Algorithms~\ref{alg:sde-poisson} and~\ref{alg:volterra-poisson} follow the theoretical scheme exactly:
the physical jump times are random, but the Gaussian jump sizes always have covariance $\eps I_m$.
In particular, this is \emph{not} the same as running Euler--Maruyama on a random time grid with variance proportional to the waiting time.

\subsection{Linear SDE with a singular time-dependent drift}

We first consider the one-dimensional linear SDE
\[
\dif X_t=\mu(t)X_t\,\dif t+\sigma(t)X_t\,\dif W_t.
\]
Its explicit solution is
\[
X_t=X_0\exp\Bigg\{\int_0^t\Big(\mu(s)-\tfrac12\sigma(s)^2\Big)\dif s+\int_0^t\sigma(s)\,\dif W_s\Bigg\}.
\]
In particular, using the exponential martingale property (with deterministic $\sigma$), we have
\[
\mE\!\left(X_t\mid\mathcal F_0\right)=X_0\exp\Big\{\int_0^t \mu(s)\,\dif s\Big\},
\]
and hence
\[
\mE(X_t)=\mE(X_0)\exp\Big\{\int_0^t \mu(s)\,\dif s\Big\},\qquad
\mE|X_t|^2=\mE|X_0|^2\exp\Big\{\int_0^t\big(2\mu(s)+\sigma(s)^2\big)\,\dif s\Big\}.
\]

In our numerical tests, we choose
\begin{align}\label{Para}
\sigma(s)\equiv\sigma_0,\qquad
\mu(s)=
\begin{cases}
\mu_0|s-s_0|^{-\alpha}, & s\in(0,\tfrac12),\\[4pt]
\mu_1|s-s_1|^{-\beta}, & s\in[\tfrac12,1],
\end{cases}
\end{align}
where $s_0\in(0,\tfrac12)$, $s_1\in(\tfrac12,1)$ and $\sigma_0,\mu_0,\mu_1\in\R$. 
Note that $\mu$ has integrable singularities at $s=s_0$ and $s=s_1$ (when $\alpha,\beta\in(0,1)$), and it also has a jump discontinuity at $s=\tfrac12$ whenever the one-sided limits differ.

We compare the Poisson approximation and Euler--Maruyama schemes for \eqref{Para} with the parameters
\[
\sigma_0=0.25,\quad \mu_0=0.3,\quad \mu_1=-0.3,\quad
s_0=0.2,\quad s_1=0.8,\quad \alpha=\beta=0.48.
\]
Figure~\ref{fig:poisson} reports the numerical results with $\eps=10^{-3}$. 
The green curves correspond to the analytic values of $t\mapsto \mE(X_t)$ and $t\mapsto \mE|X_t|^2$ computed from the explicit formulas above. 
The blue and red curves are the empirical means obtained from $4000$ Monte Carlo samples, produced by the Poisson approximation and Euler--Maruyama schemes, respectively.

\begin{figure}[ht]
\centering
\IfFileExists{Poisson_Euler0.png}{\includegraphics[width=4in,height=2in]{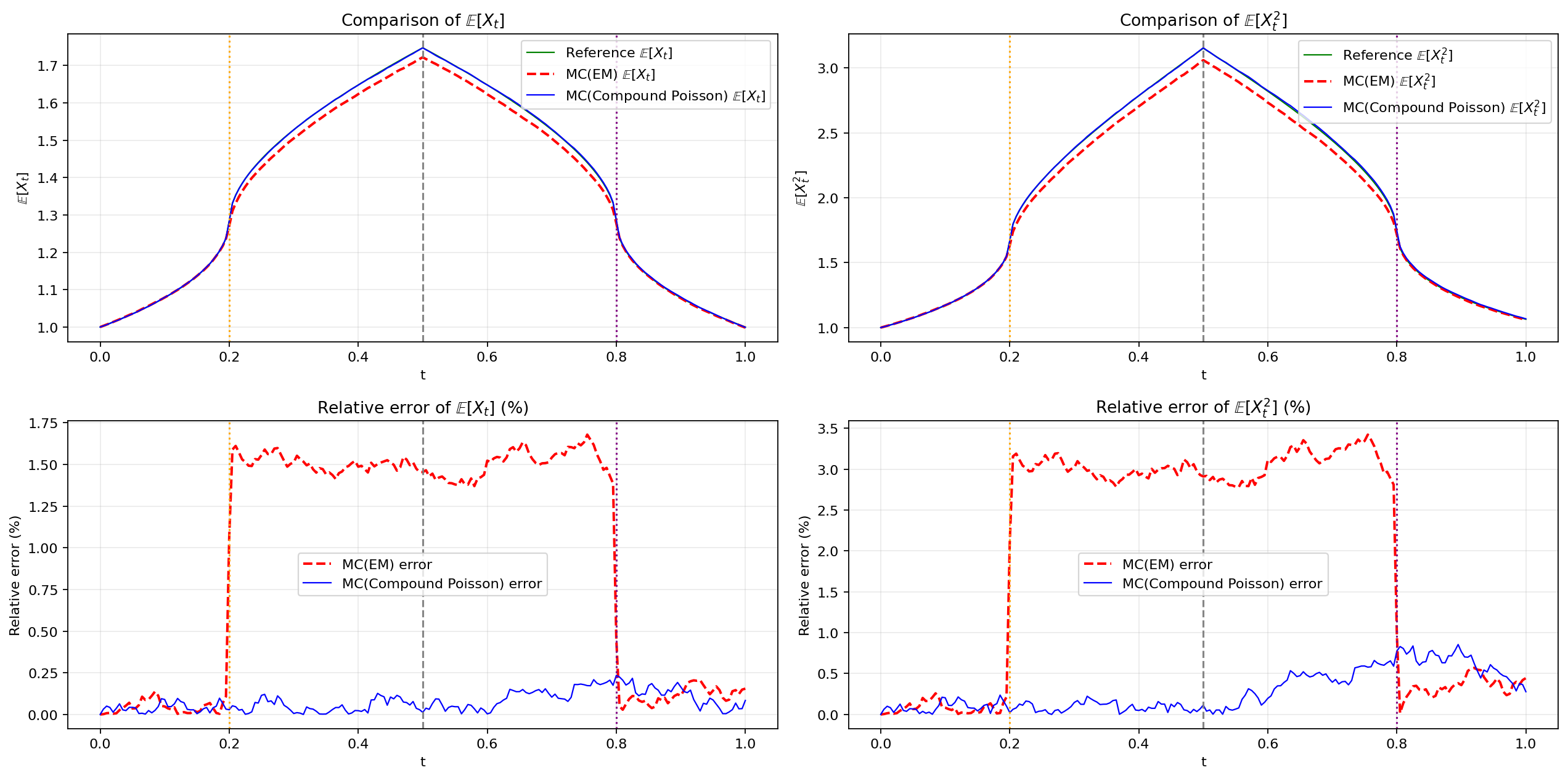}}{\fbox{\texttt{\detokenize{Poisson_Euler0.png}}}}
\caption{Approximations of $t\mapsto \mE(X_t)$ and $t\mapsto \mE|X_t|^2$ for the linear SDE \eqref{Para}.}
\label{fig:poisson}
\end{figure}

\subsection{Linear stochastic Volterra equation with a singular kernel}
Next, we consider the linear stochastic Volterra equation
\[
X_t=X_0
+\mu\int_0^t (t-s)^{-\alpha_0}|s-s_0|^{-\beta_0}X_s\,\dif s
+\sqrt{\sigma}\int_0^t (t-s)^{-\alpha_1/2}|s-s_1|^{-\beta_1/2}X_s\,\dif W_s,
\]
where $\alpha_0,\beta_0,\alpha_1,\beta_1,s_0,s_1\in[0,1)$, $\mu\in\R$, and $\sigma\ge0$.

For $\alpha,\beta,a\in[0,1)$ and $\mu\in\R$, define the Volterra operator
\[
(K^{\mu,a}_{\alpha,\beta}f)(t)
:=\mu\int_0^t (t-s)^{-\alpha}|s-a|^{-\beta}f(s)\,\dif s,
\qquad f:[0,1]\to\R.
\]

\begin{lemma}
\label{lem:neumann}
The mean $m(t):=\mE(X_t)$ satisfies the linear Volterra equation
\[
m(t)=\mE(X_0) + (K^{\mu,s_0}_{\alpha_0,\beta_0}m)(t),
\]
and admits the (convergent) Neumann series representation
\begin{equation}\label{As1}
\mE(X_t)=\mE(X_0)\Bigg(1+\sum_{n=1}^{\infty}\Big[(K^{\mu,s_0}_{\alpha_0,\beta_0})^{n}\mathbf{1}\Big](t)\Bigg).
\end{equation}
If $\mu=0$, then $v(t):=\mE|X_t|^2$ satisfies
\[
v(t)=\mE|X_0|^2 + (K^{\sigma,s_1}_{\alpha_1,\beta_1}v)(t),
\]
and hence
\begin{equation}\label{As2}
\mE|X_t|^2=\mE|X_0|^2\Bigg(1+\sum_{n=1}^{\infty}\Big[(K^{\sigma,s_1}_{\alpha_1,\beta_1})^{n}\mathbf{1}\Big](t)\Bigg).
\end{equation}
\end{lemma}

\begin{proof}
Taking expectations in the Volterra equation and using $\mE\int_0^t \cdots \dif W_s=0$, we obtain
\[
\mE(X_t)=\mE(X_0)+\mu\int_0^t (t-s)^{-\alpha_0}|s-s_0|^{-\beta_0}\mE(X_s)\,\dif s
=\mE(X_0) + (K^{\mu,s_0}_{\alpha_0,\beta_0}m)(t).
\]
Its convergence on any finite time interval follows from the standard resolvent theory for Volterra integral equations: 
under $\alpha,\beta\in[0,1)$ the kernel $(t,s)\mapsto (t-s)^{-\alpha}|s-a|^{-\beta}$ is locally integrable on $\{0\le s\le t\le 1\}$, hence the associated Volterra operator admits a resolvent kernel and the corresponding Neumann (resolvent) series converges; see, e.g., \cite[Chapter~2]{gripenberg1990volterra}.
The same argument applies to \eqref{As2} when $\mu=0$, since It\^o's isometry gives
\[
\mE|X_t|^2=\mE|X_0|^2+\sigma\int_0^t (t-s)^{-\alpha_1}|s-s_1|^{-\beta_1}\mE|X_s|^2\,\dif s
=\mE|X_0|^2+(K^{\sigma,s_1}_{\alpha_1,\beta_1}v)(t).
\]
This completes the proof.
\end{proof}

In our experiments we use the parameter sets
\begin{align*}
&X_0=1,\quad \mu=0.2,\quad \sigma=0.1,\quad
\alpha_0=0.3,\quad \beta_0=0.5,\quad
\alpha_1=0.2,\quad \beta_1=0.4,\quad
s_0=0.2,\quad s_1=0;\\
&X_0=1,\quad \mu=0,\quad \sigma=0.3,\quad
\alpha_1=0.05,\quad \beta_1=0.25,\quad s_1=0.4.
\end{align*}
Figure~\ref{fig:volterra} shows the numerical results for $t\mapsto \mE(X_t)$ and $t\mapsto \mE|X_t|^2$.
The green curves are reference values computed by truncating the series representations \eqref{As1} and \eqref{As2} at a sufficiently large order.
The blue and red curves are empirical means from $10{,}000$ sample paths generated by the Poisson approximation and Euler--Maruyama schemes, respectively.
The plots indicate that the Poisson approximation achieves visibly higher accuracy in the presence of singular Volterra kernels.

\begin{figure}[ht]
\centering
\IfFileExists{Volterra_Poisson_Euler.png}{\includegraphics[width=4in,height=2in]{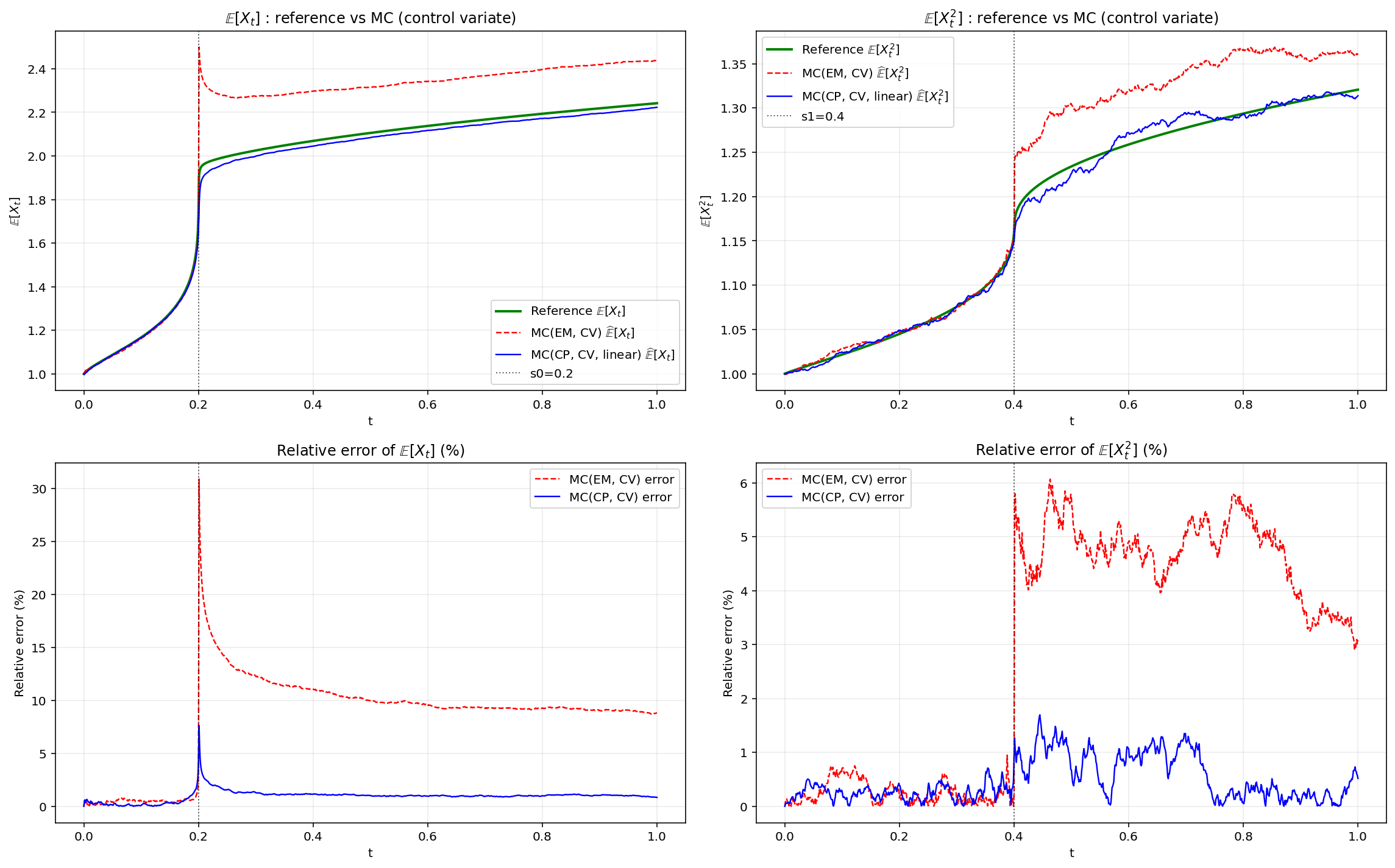}}{\fbox{\texttt{\detokenize{Volterra_Poisson_Euler.png}}}}
\caption{Approximations of $t\mapsto \mE(X_t)$ (left) and $t\mapsto \mE|X_t|^2$ (right) for the stochastic Volterra equation.}
\label{fig:volterra}
\end{figure}

\end{document}